\documentclass[11pt,twoside,a4paper]{article}
\usepackage[T1]{fontenc}
\usepackage[utf8]{inputenc}

\usepackage[australian]{babel}

\usepackage[driver=pdftex,margin=3cm,heightrounded=true,centering,includeheadfoot]{geometry}

\usepackage{hyperref}

\usepackage[intlimits,tbtags]{mathtools} % Passes options to amsmath
\usepackage{amssymb,amsfonts}
\usepackage{amsthm}
\usepackage[all]{xy}
\usepackage[capitalise]{cleveref}

\allowdisplaybreaks
\numberwithin{equation}{section}

\usepackage[dvipsnames]{xcolor}
\definecolor{MyBlue}{cmyk}{1,0.13,0,0.63}
\definecolor{MyGreen}{cmyk}{0.91,0,0.88,0.52}
\definecolor{MyRed}{rgb}{.6,0,0}
\newcommand{\mylinkcolor}{MyBlue}
\newcommand{\mycitecolor}{MyGreen}
\newcommand{\myurlcolor}{MyRed}
\theoremstyle{plain}
\newtheorem{thm}{Theorem}[section]
\newtheorem{lem}[thm]{Lemma}
\newtheorem{prop}[thm]{Proposition}
\newtheorem{coro}[thm]{Corollary}
\theoremstyle{definition}
\newtheorem{defn}[thm]{Definition}
\newtheorem{remark}[thm]{Remark}
\newtheorem{example}[thm]{Example}
\newtheoremstyle{myAssumption}
  {\topsep}   % ABOVESPACE
  {\topsep}   % BELOWSPACE
  {\normalfont}  % BODYFONT
  {0pt}       % INDENT (empty value is the same as 0pt)
  {\bfseries} % HEADFONT
  {.}         % HEADPUNCT
  {5pt plus 1pt minus 1pt} % HEADSPACE
  {\thmname{#1}\thmnote{ #3}}          % CUSTOM-HEAD-SPEC, eg: \thmname{#1}\thmnumber{ #2}\thmnote{ (#3)}
\theoremstyle{myAssumption}
\newtheorem*{assumption*}{Assumption}

\makeatletter
\newcommand{\customlabel}[2]{%
   \protected@write \@auxout {}{\string \newlabel {#1}{{#2}{\thepage}{#2}{#1}{}} }%
   \hypertarget{#1}{}
}
\makeatother

\renewcommand{\eqref}[1]{\labelcref{#1}}
\crefname{thm}{Theorem}{Theorems}
\crefname{lem}{Lemma}{Lemmas}
\crefname{prop}{Proposition}{Propositions}
\crefname{coro}{Corollary}{Corollaries}
\crefname{defn}{Definition}{Definitions}
\crefname{remark}{Remark}{Remarks}
\crefname{example}{Example}{Examples}

\hypersetup{%
  bookmarksnumbered=true,bookmarksopen=false,%
  plainpages=false,% necessary to prevent duplicate page identifiers
  linktocpage=true,%
  colorlinks=true,breaklinks=true,%
  linkcolor=\mylinkcolor,citecolor=\mycitecolor,urlcolor=\myurlcolor,%
  pdfpagelayout=OneColumn,%
  pageanchor=true,%
}

\makeatletter
\def\@endtheorem{\endtrivlist}% NEW
\def\thm@space@setup{%
  \thm@preskip=4pt plus 2pt minus 2pt
  \thm@postskip=\thm@preskip
}
\renewenvironment{proof}[1][\proofname]{\par
  \pushQED{\qed}%
  \normalfont \topsep4\p@\relax % NEW
  \trivlist
  \item[\hskip\labelsep
        \itshape
    #1\@addpunct{.}]\ignorespaces
}{%
  \popQED\endtrivlist\@endpefalse
}
\makeatother

\usepackage[shortlabels]{enumitem}
\setlist{topsep=4pt plus 2pt minus 2pt,partopsep=0pt,itemsep=2pt plus 2pt minus 2pt,parsep=0.5\parskip}
\setenumerate[1]{label=(\arabic*),font=\upshape}

\usepackage{fancyhdr}
\newcommand{\MR}[1]{}

\let\OLDthebibliography\thebibliography
\renewcommand\thebibliography[1]{
  \addcontentsline{toc}{section}{\refname}
  \OLDthebibliography{#1}
  \setlength{\parskip}{0pt}
  \setlength{\itemsep}{0pt plus 0.3ex}
}

\newcommand{\R}{\mathbb{R}}
\newcommand{\C}{\mathbb{C}}
\newcommand{\Z}{\mathbb{Z}}
\newcommand{\bD}{\mathbb{D}}
\newcommand{\A}{\mathcal{A}}
\newcommand{\B}{\mathcal{B}}
\newcommand{\E}{\mathcal{E}}
\newcommand{\mH}{\mathcal{H}}
\newcommand{\mL}{\mathcal{L}}
\newcommand{\mK}{\mathcal{K}}
\newcommand{\mN}{\mathcal{N}}
\newcommand{\D}{\mathcal{D}}
\newcommand{\pS}{\mathcal{S}}
\newcommand{\pT}{\mathcal{T}}
\newcommand{\pF}{\mathcal{F}}

\newcommand{\gradS}{\Gamma_\pS}
\newcommand{\gradF}{\Gamma_\pF}
\newcommand{\gradD}{\Gamma_\D}

\DeclareMathOperator{\Dom}{Dom}
\DeclareMathOperator{\Ran}{Ran}
\DeclareMathOperator{\Index}{Index}
\DeclareMathOperator{\ind}{ind}
\DeclareMathOperator{\relind}{rel-ind}
\DeclareMathOperator{\End}{End}
\DeclareMathOperator{\Ker}{Ker}
\DeclareMathOperator{\Id}{Id}

\DeclareMathOperator{\Span}{span}
\DeclareMathOperator{\supp}{supp}

\renewcommand{\bar}[1]{\overline{#1}}
\newcommand{\K}{\textnormal{K}}
\newcommand{\KK}{\textnormal{KK}}
\newcommand{\SF}{\textnormal{sf}}
\newcommand{\til}[1]{\widetilde{#1}}
\newcommand{\into}{\hookrightarrow}
\newcommand{\la}{\langle}
\newcommand{\ra}{\rangle}

\newcommand{\bF}{\mathtt{F}}

\newcommand{\mattwo}[4]{
  \begin{pmatrix}#1&#2\\ #3&#4\end{pmatrix}
}
\newcommand{\matfour}[4]{
  \begin{pmatrix}#1\\ #2\\ #3\\ #4\end{pmatrix}
}

\def\myTitle{Dirac--Schrödinger operators, index theory, and spectral flow}

\title{Dirac--Schrödinger operators,\\ index theory, and spectral flow}
\author{
Koen van den Dungen%
\footnote{Email: kdungen@uni-bonn.de}
\\[2mm]
{\small Mathematisches Institut}, 
{\small Universität Bonn}\\
{\small Endenicher Allee 60, D-53115 Bonn}
}

\date{}

\begin{document}

\maketitle

\begin{abstract}
\noindent
In this article we study generalised Dirac--Schrödinger operators in arbitrary signatures (with or without gradings), providing a general $\KK$-theoretic framework for the study of index pairings and spectral flow. 
We provide a general Callias Theorem, which shows that the index (or the spectral flow, or abstractly the $\K$-theory class) of Dirac--Schrödinger operators can be computed on a suitable compact hypersurface. 
Furthermore, if the zero eigenvalue is isolated in the spectrum of the Dirac operator, we relate the index (or spectral flow) of Dirac--Schrödinger operators to the index (or spectral flow) of corresponding Toeplitz operators. 
Combining both results, we obtain an index (or spectral flow) equality relating Toeplitz operators on the noncompact manifold to Toeplitz operators on the compact hypersurface. 
Our results generalise various known results from the literature, while presenting these results in a common unified framework. 

\vspace{\baselineskip}

\end{abstract}

\section{Introduction}

A Dirac--Schrödinger operator on a (typically noncompact, complete) Riemannian manifold $M$ consists of a Dirac-type operator $\D$ together with a \emph{potential} $\pS$. 
Classically, the potential is given by a self-adjoint endomorphism on some auxiliary vector bundle (of finite rank) over $M$. 
Most commonly, one considers an \emph{odd}-dimensional manifold $M$ (for which the Dirac operator is \emph{ungraded}) along with an \emph{ungraded} potential. 
Nevertheless, some results have appeared in the literature also for even-dimensional manifolds (for which the Dirac operator is naturally \emph{$\Z_2$-graded}). 
Altogether, we may distinguish between the following four cases, labelled by their signature $(p,q)$ with $p,q\in\{0,1\}$:
\begin{itemize}
\item the odd-odd signature $\boldsymbol{(p,q)=(1,1)}$: $\pS$ and $\D$ are both ungraded;
\item the even-even signature $\boldsymbol{(p,q)=(0,0)}$: $\pS$ and $\D$ are both $\Z_2$-graded;
\item the even-odd signature $\boldsymbol{(p,q)=(0,1)}$: $\pS$ is $\Z_2$-graded and $\D$ is ungraded;
\item the odd-even signature $\boldsymbol{(p,q)=(1,0)}$: $\pS$ is ungraded and $\D$ is $\Z_2$-graded. 
\end{itemize}
In each of these signatures, we will denote the corresponding Dirac--Schrödinger operator by $\pS\times\D$; its precise construction depends on the signature $(p,q)$ and will be described in \cref{defn:gen_DS} below. 

In this paper, our aim is to prove the following three types of results on Dirac--Schrödinger operators, in a rather general setting, and in particular for each of the four possible signatures $(p,q)$: 
\begin{description}
\item[Pairings:] 
A Dirac--Schrödinger operator $\pS\times\D$ is Fredholm, and it represents (via its index, or via its spectral flow, or abstractly via its class in $\K$-theory) the pairing of the $\K$-theory class of the potential $\pS$ with the $\K$-homology class of $\D$. 
\item[Callias Theorems:] 
The index (or spectral flow, or class in $\K$-theory) of a Dirac--Schrö\-dinger operator can be computed on a suitable compact hypersurface. 
\item[Toeplitz Theorems:] The Dirac--Schrödinger operator $\pS\times\D$ and its associated Toeplitz operator, obtained by compressing the potential $\pS$ to the kernel of $\D$, have the same index (or the same spectral flow, or the same class in $\K$-theory). 
\end{description}
Such results have previously appeared in various forms in the literature, and in various of the four signatures described above. 
For instance, Bunke \cite{Bun95} first showed that Dirac--Schrödinger operators represent pairings of $\K$-theory with $\K$-homology in both the odd-odd signature and the even-even signature. The pairing in the odd-odd signature was also generalised to potentials consisting of families of \emph{unbounded} self-adjoint operators by Kaad and Lesch \cite[\S8]{KL13} (see also \cite{vdD19_Index_DS}). 
Abstract versions of these pairings (on spectral triples) in the odd-odd and even-even signatures have recently been studied in \cite{SS23}. 

The classical Callias Theorem is well-established in the odd-odd signature \cite{Cal78,Ang90,BM92,Ang93a,Rad94,Bun95,Kuc01,Gesztesy-Waurick16}, 
and a generalisation for potentials consisting of families of \emph{unbounded} operators is given in \cite{vdD25_Callias}. 
An analogue of the Callias Theorem in the even-even signature was described by Bunke \cite[\S2.4]{Bun95}. 

Toeplitz theorems were given in the even-even signature for the \emph{index} of Toeplitz operators \cite{GH96,Bun00} and in the odd-even signature for the \emph{spectral flow} of a family of Toeplitz operators by Braverman \cite{Bra19}, with applications to topological insulators and the bulk-edge correspondence. 

Our main goal in this paper is thus to present a single unified framework, which not only neatly incorporates all of the aforementioned results, but also extends these results to a more general setting and to all four signatures. 
In particular, we mention here a few new results that emerge from our general framework:
\begin{itemize}
\item 
We provide a Callias Theorem for the even-odd and odd-even signatures, which are particularly interesting in the context of spectral flow. 
\item 
We prove a new Toeplitz index theorem (in the even-even signature) as well as a Toeplitz spectral flow theorem (in the odd-even signature), relating the index or spectral flow of Toeplitz operators on a noncompact manifold with Toeplitz operators on a compact hypersurface. 
\end{itemize}

Let us briefly introduce our framework and summarise our main results. 
Let $M$ be a connected Riemannian manifold, and let $\D$ be an essentially self-adjoint elliptic first-order differential operator on $M$ (for instance, a Dirac-type operator).
Moreover, we fix a $C^*$-algebra $A$, and we consider a family of regular self-adjoint operators $\{\pS(x)\}_{x\in M}$ with compact resolvents and with common domain on a Hilbert $C^*$-module $E$ over $A$, such that $\{\pS(x)\}_{x\in M}$ is uniformly invertible outside a compact subset $K\subset M$ (see assumption \ref{ass:A} in \cref{sec:gen_DS}). 
Furthermore, we assume that $\pS(\cdot)$ is weakly differentiable and that $\big[\D,\pS(\cdot)\big] \big(\pS(\cdot)\pm i\big)^{-1}$ is well-defined and bounded (see assumption \ref{ass:B} in \S\ref{sec:gen_DS_Fred}). 
We provide in \cref{sec:gen_DS} a generalised construction of a \emph{Dirac--Schrödinger operator} $\pS\times\D$ in each of the four signatures $(p,q)$. 
We continue in \S\ref{sec:rel_ind_thm} to generalise the $\K$-theoretic relative index theorem \cite[Theorem 4.1]{vdD25_Callias} (which goes back to \cite{Bun95}) to arbitrary signatures. 
In \S\ref{sec:gen_DS_Fred} and \S\ref{sec:gen_DS_Kasp_prod} we prove that the Dirac--Schrödinger operator $\pS\times\D$ is Fredholm, and that it represents the $\K_{p+q}(A)$-valued pairing of the $\K$-theory class of the potential $\pS(\cdot)$ with the $\K$-homology class of $\D$: 
\[
\begin{array}{rclcl}
\K_p\big(C_0(M,A)\big) &\times& \K^q\big(C_0(M)\big) &\to& \K_{p+q}(A) , \\
{}[\pS(\cdot)] &\times& [\D] &\mapsto& [\pS\times\D] . 
\end{array}
\]
Most of the literature on this topic considers the case $A=\C$. Then, for the odd-odd and even-even signatures, the class $[\pS\times\D] \in \K_0(\C)$ corresponds to the Fredholm index of $(\pS\times\D)_+$ in $\K_0(\C) \simeq \Z$. 
For the odd-even and even-odd signatures, for which $p+q=1$, the pairing $\K_p(C_0(M)) \times \K^q(C_0(M)) \to \K_1(\C) = \{0\}$ is of course trivial. 
However, in these signatures one can still obtain a nontrivial pairing by considering instead the $C^*$-algebra $A=C(S^1)$, and this is in fact one of our main reasons for including this auxiliary $C^*$-algebra $A$. 
For instance, the setting under consideration in \cite{Bra19} takes $A=C(S^1)$ with the odd-even signature $(p,q)=(1,0)$, and the pairing then corresponds to the \emph{spectral flow} of the family of Dirac--Schrödinger operators $\{\pS_t\times\D\}_{t\in S^1}$ on $M$: 
\begin{equation*}
\begin{array}{rclcl}
\K_1(C_0(M\times S^1)) &\times& \K^0(C_0(M)) &\to& \K_1(C(S^1)) \simeq \Z , \\
{}\big[\{\pS_t(\cdot)\}_{t\in S^1}\big] &\times& [\D] &\mapsto& \SF\left( \{\pS_t\times\D\}_{t\in S^1} \right) .
\end{array}
\end{equation*}
A similar approach would be possible also in the even-odd signature (though the author is not aware of this case being present in the literature). 

Next, in \cref{sec:gen_Callias} we provide a Callias theorem, generalising the main result of \cite{vdD25_Callias} to arbitrary signatures. 
We recall that the potential $\pS(\cdot)$ is assumed to be invertible outside of a compact subset $K\subset M$; we now assume that $K$ has a smooth boundary $N = \partial K$. 
The \emph{Callias theorem} (\cref{thm:Callias} below) then states that, under suitable assumptions, the class of the Dirac--Schrödinger operator can be computed from a pairing on the compact hypersurface $N$:
\begin{align*}
[\pS\times\D]
= \relind_{p+1}\big(P_+(\pS_N(\cdot)),P_+(\pT(\cdot))\big) \otimes_{C(N)} [\D_N] \in \K_{p+q}(A) .
\end{align*}
Here $\relind_{p+1}$ denotes the \emph{even} (if $p+1=0$) or \emph{odd} (if $p+1=1$) relative index of two projections. 
The even relative index is the `usual' relative index of pairs of projections. 
The case $p=0$ requires us to consider also the \emph{odd} relative index, as described in \cite[\S8.1]{Wah07} and briefly recalled in \S\ref{sec:rel-ind}. 

In \S\ref{sec:classical_Callias}, we specialise to the case where $A$ is unital and $E$ is finitely generated and projective. 
To emphasise this restriction, we will denote the potential by $\pF$ instead of $\pS$. 
In the odd-odd signature, the Callias Theorem yields a pairing between the even $\K$-theory $\K_0(C(N,A))$ and the even $\K$-homology $\K^0(C(N))$, and we will show that, as in the classical Callias Theorem, the index of the Dirac--Schrödinger operator can be computed via the \emph{even index pairing} on $N$ (see \cref{thm:11_index}):
\[
[\pF\times\D] 
= \big[ P_+(\pF_N) \big] \otimes_{C(N)} [\D_N] 
= \Index\big( P_+(\pF_N) \D_N P_+(\pF_N) \big) \in \K_0(A) .
\]
Furthermore, in the even-even signature, we have instead a pairing between the odd $\K$-theory $\K_1(C(N,A))$ and odd $\K$-homology $\K^1(C(N))$, which can now be computed in terms of the \emph{odd index pairing} or the \emph{spectral flow pairing} on $N$ (see \cref{thm:00_sf}, generalising a result from \cite[Theorem 2.16]{Bun95}). 
For nontrivial $A$, we can also consider the odd-even pairing taking values in $\K_1(A)$. In particular, if $A = C(S^1,B)$ for some unital $C^*$-algebra $B$, then the odd-even pairing can be interpreted as a spectral flow taking values in $\K_0(B)$ (see \cref{coro:10_sf}). 

As a consequence of the Callias Theorem, one can formulate the cobordism invariance of the index, as we describe in \S\ref{sec:cobordism}. 
Although one might not really need our generalisation of the Callias Theorem to prove cobordism invariance, it is nevertheless an advantage of our approach, considering all possible signatures simultaneously, that we can easily formulate cobordism invariance results also in other signatures. As an example, we present the cobordism invariance of the spectral flow pairing in \cref{coro:cobordism_sf}. 

In \cref{sec:Toeplitz}, we continue with the special case where $A$ is unital and $E$ is finitely generated and projective, and we again denote the potential by $\pF$ (instead of $\pS$). 
Our main additional assumption is now that zero is an isolated point of the spectrum of $\D$. 
We can then consider the \emph{Toeplitz operator}, given by the potential $\pF$ compressed to the kernel of $\D$: 
\[
T_\pF = P_0 \circ (\pF\times\D) \circ P_0 ,
\]
where $P_0$ denotes the projection onto the kernel of $\D$. 
Our final main result is that the Dirac--Schrödinger operator $\pF\times\D$ and its corresponding Toeplitz operator $T_\pF$ represent the same class in $\K$-theory: 
\[
[T_\pF] = [\pF\times\D] \in \K_{p+q}(A) .
\]
We will see that the assumption on $\D$ implies that this class is in fact trivial in the case $q=1$. 
However, as described in \S\ref{sec:Toeplitz_00} and \S\ref{sec:Toeplitz_10}, we obtain a nontrivial index theorem in the even-even signature $(p,q)=(0,0)$ (generalising results from \cite{GH96,Bun00}, where $A=\C$), as well as a nontrivial spectral flow theorem in the odd-even signature $(p,q)=(1,0)$ (generalising a result from \cite{Bra19}, where $A=C(S^1)$). 

In \S\ref{sec:Toeplitz_Callias} we combine our results on Toeplitz operators with the Callias Theorem. 
This allows us to obtain equalities for the index (or spectral flow) of Toeplitz operators on the noncompact manifold $M$ and on the compact hypersurface $N$. 
In the even-even signature $(p,q)=(0,0)$ we obtain the following \emph{index theorem for Toeplitz operators} (see \cref{coro:Toeplitz_index}):
\[
\Index\big( P_0(\D_+) \pF_+ P_0(\D_+) \big) - \Index\big( P_0(\D_-) \pF_+ P_0(\D_-) \big)
= \Index\big( P_+(\D_N) {\pF_+}|_N P_+(\D_N) \big) . 
\]
Similarly, in the odd-even signature $(p,q)=(1,0)$ we obtain a \emph{spectral flow theorem for Toeplitz operators} (see \cref{coro:Toeplitz_sf}), relating families of Toeplitz operators on $M$ and on the hypersurface $N$:
\begin{multline*}
\SF\big( \big\{P_0(\D_+) \pF_t P_0(\D_+) \big\}_{t\in S^1} \big) - \SF\big( \big\{P_0(\D_-) \pF_t P_0(\D_-) \big\}_{t\in S^1} \big) \\
= \SF\big( \big\{ P_+(\pF_N(t)) \D_N P_+(\pF_N(t)) \big\}_{t\in S^1} \big) .
\end{multline*}
We illustrate both these results by the example of the Dolbeault--Dirac operator on a strongly pseudoconvex domain in \S\ref{sec:pseudoconvex}. 

Finally, Appendix \ref{sec:prelim} collects several facts regarding Fredholm operators on Hilbert modules, recalling in particular the definitions and properties of the (even or odd) relative index of a pair of projections and of the (even or odd) spectral flow of a path of Fredholm operators. 
Appendix \ref{sec:index_pairings} describes the pairing (via the Kasparov product) of $\K$-theory with $\K$-homology in the presence of an auxiliary $C^*$-algebra $A$. In particular, we adapt the well-known descriptions of the even and odd index pairings in order to also describe the even-odd and odd-even pairings (which may be nonzero in the presence of a nontrivial $C^*$-algebra $A$).

\subsection*{Notation}

Throughout this paper, let $A$ be a $\sigma$-unital $C^*$-algebra, and let $E$ be a (possibly $\Z_2$-graded) countably generated Hilbert $C^*$-module over $A$ (or Hilbert $A$-module for short) with $A$-valued inner product $\la\cdot|\cdot\ra$. 
(The reader unfamiliar with $C^*$-modules may consider the special case $A=\C$, so that $E$ is simply a separable Hilbert space. For an introduction to Hilbert $C^*$-modules, we refer to \cite{Lance95}.) 
The norm of $\psi\in E$ is given by $\|\psi\| = \|\la\psi|\psi\ra\|^{\frac12}$. 

The space of adjointable linear operators $E\to E$ is denoted by $\mL_A(E)$.
For any $\psi,\eta\in E$, the rank-one operators $\theta_{\psi,\eta}$ are defined by $\theta_{\psi,\eta}\xi := \psi \la\eta|\xi\ra$ for $\xi\in E$. 
The compact endomorphisms $\mK_A(E)$ are given by the closure of the space of finite linear combinations of rank-one operators. 
For two Hilbert $A$-modules $E_1$ and $E_2$, the adjointable linear operators $E_1\to E_2$ are denoted by $\mL_A(E_1,E_2)$. 

A densely defined operator $S$ is called \emph{regular} if $S$ is closed, the adjoint $S^*$ is densely defined, and $1+S^*S$ has dense range (note that on a Hilbert space, every closed operator is regular). 
A densely defined, closed, symmetric operator $S$ is regular and self-adjoint if and only if the operators $S\pm i$ are surjective \cite[Lemma 9.8]{Lance95}.

\subsection*{Acknowledgements}

The author thanks Matthias Lesch for interesting discussions and helpful suggestions. 
The author acknowledges financial support from the Hausdorff Center for Mathematics in Bonn, funded by the Deutsche Forschungsgemeinschaft (DFG, German Research Foundation) under Germany's Excellence Strategy -- EXC-2047/1 -- 390685813.

\section{Dirac--Schrödinger operators}
\label{sec:gen_DS}

In this section, we introduce our generalised construction of Dirac--Schrödinger operators, following the definitions of \cite[\S3.1]{vdD25_Callias}, but allowing for arbitrary signatures and $\Z_2$-gradings (the setting of \cite{vdD25_Callias} corresponds to the odd-odd signature described below). 

\begin{assumption*}[(A)]
\customlabel{ass:A}{(A)}
Let $A$ be a (trivially graded) $\sigma$-unital $C^*$-algebra, and let $E$ be a countably generated Hilbert $A$-module. 
Let $M$ be a connected Riemannian manifold (typically noncompact), and let $\D$ be an essentially self-adjoint elliptic first-order differential operator on a hermitian vector bundle $\bF\to M$.
Let $\{\pS(x)\}_{x\in M}$ be a family of regular self-adjoint operators with common domain on $E$ satisfying the assumptions
\begin{itemize}
\item[(A1)]
\customlabel{ass:A1}{(A1)}
The domain $W := \Dom\pS(x)$ is independent of $x\in M$, and the inclusion $W\into E$ is a compact endomorphism between Hilbert $A$-modules (where $W$ is viewed as a Hilbert $A$-module equipped with the graph norm of $\pS(x_0)$, for some $x_0\in M$). 
\item[(A2)]
\customlabel{ass:A2}{(A2)}
The map $\pS\colon M\to\mL_A(W,E)$ is norm-continuous. 
\item[(A3)]
\customlabel{ass:A3}{(A3)}
There is a compact subset $K\subset M$ such that $\pS(x)$ is uniformly invertible on $M\setminus K$. 
\end{itemize}
Furthermore, let $\gradS$ be a self-adjoint unitary operator on the Hilbert $A$-module $E$ (preserving the domain of $\pS(\cdot)$), and let $\gradD$ be a self-adjoint unitary smooth endomorphism on the bundle $\bF\to M$. 
We consider the following four different cases, labelled by their \emph{signatures} $(p,q) \in \Z_2\times\Z_2$:
\begin{description}
\item[the odd-odd signature $\boldsymbol{(p,q)=(1,1)}$:] $\gradS=\Id_E$ and $\gradD=\Id_\bF$; 
\item[the even-even signature $\boldsymbol{(p,q)=(0,0)}$:] $\pS(x)$ anti-commutes with $\gradS$ (for all $x\in M$), and $\D$ anti-commutes with $\gradD$; 
\item[the odd-even signature $\boldsymbol{(p,q)=(1,0)}$:] $\gradS=\Id_E$, and $\D$ anti-commutes with $\gradD$; 
\item[the even-odd signature $\boldsymbol{(p,q)=(0,1)}$:] $\pS(x)$ anti-commutes with $\gradS$ (for all $x\in M$), and $\gradD=\Id_\bF$. 
\end{description}
\end{assumption*}
By convention, we will always compute $p$ and $q$ modulo $2$ (e.g., if $p=1$, then $p+1\equiv0$). 
In the following, we will consider all four signatures $(p,q)$ simultaneously. 

Given the family of operators $\{\pS(x)\}_{x\in M}$ on $E$, we obtain a closed symmetric operator $\pS(\cdot)$ on $C_0(M,E)$, which is defined as the closure of the operator $\big(\pS(\cdot)\psi\big)(x) := \pS(x) \psi(x)$ on the initial dense domain $C_c(M,W)$. 
The operator $\pS(\cdot)$ on the Hilbert $C_0(M,A)$-module $C_0(M,E)$ is regular self-adjoint and Fredholm by \cite[Proposition 3.4]{vdD19_Index_DS}, so by \cref{prop:Fred_KK} it defines a class 
\[
[\pS(\cdot)] \in \KK^p(\C,C_0(M,A)) \simeq \K_p(C_0(M,A)) .
\]
Furthermore, since $\D$ is an essentially self-adjoint first-order differential operator, and since the ellipticity of $\D$ ensures that $\D$ also has locally compact resolvents \cite[Proposition 10.5.2]{Higson-Roe00}, we know that $(C_0^1(M), L^2(M,\bF), \D)$ is an (even or odd) spectral triple, which represents a $\K$-homology class 
\[
[\D] \in \KK^q(C_0(M),\C) \equiv \K^q(C_0(M)) . 
\]

We consider the balanced tensor product $L^2(M,E\otimes\bF) := C_0(M,E) \otimes_{C_0(M)} L^2(M,\bF)$. 
The operator $\pS(\cdot)\otimes1$ is well-defined on $\Dom\pS(\cdot) \otimes_{C_0(M)} L^2(M,\bF) \subset L^2(M,E\otimes\bF)$, and is denoted simply by $\pS(\cdot)$ as well. By \cite[Proposition 9.10]{Lance95}, $\pS(\cdot)$ is regular self-adjoint on $L^2(M,E\otimes\bF)$. 

The operator $1\otimes\D$ is not well-defined on $L^2(M,E\otimes\bF)$. Instead, using the canonical isomorphism $L^2(M,E\otimes\bF) \simeq E \otimes L^2(M,\bF)$, we consider the operator $1\otimes\D$ (or $\gradS\otimes\D$, if $(p,q)=(0,0)$) on $E \otimes L^2(M,\bF)$ with domain $E\otimes\Dom\D$. 
Alternatively, we can extend the exterior derivative on $C_0^1(M)$ to an operator 
\[
d \colon C_0^1(M,E) \xrightarrow{\simeq} E\otimes C_0^1(M) \xrightarrow{1\otimes d} E\otimes\Gamma_0(T^*M) \xrightarrow{\simeq} \Gamma_0(E\otimes T^*M) .
\]
Denoting by $\sigma$ the principal symbol of $\D$, we can define an operator $1\otimes_d\D$ on the Hilbert space $C_0(M,E) \otimes_{C_0(M)} L^2(M,\bF)$ by setting
\[
(1\otimes_d\D)(\xi\otimes\psi) := 
\begin{cases}
\gradS\xi\otimes\D\psi + (\gradS\otimes\sigma)(d\xi)\psi , & \quad \text{if } (p,q)=(0,0) , \\
\xi\otimes\D\psi + (1\otimes\sigma)(d\xi)\psi , & \quad \text{otherwise} .
\end{cases}
\]
We note that the case $(p,q)=(0,0)$ corresponds to the \emph{graded} tensor product. 
Under the isomorphism $C_0(M,E) \otimes_{C_0(M)} L^2(M,\bF) \simeq E \otimes L^2(M,\bF)$, the operator $1\otimes\D$ (or $\gradS\otimes\D$, if $(p,q)=(0,0)$) on $E \otimes L^2(M,\bF)$ agrees with $1\otimes_d\D$ on $C_0(M,E) \otimes_{C_0(M)} L^2(M,\bF)$. We will denote this operator on $L^2(M,E\otimes\bF)$ simply as $\D$. The operator $\D$ is regular self-adjoint on $L^2(M,E\otimes\bF)$ (see also \cite[Theorem 5.4]{KL13}). 

Finally, we also consider on $L^2(M,E\otimes\bF)$ the self-adjoint unitaries 
\[
\gradS \equiv \gradS\otimes1
\quad\text{and}\quad 
\gradD \equiv 1\otimes\gradD . 
\]
We define the Hilbert $A$-module $\til E$ as 
\begin{equation}
\label{eq:tilE}
\til E := 
\begin{cases}
L^2(M,E\otimes\bF)^{\oplus2} , \quad&\text{if } (p,q)=(1,1) , \\ 
L^2(M,E\otimes\bF) , \quad&\text{otherwise.}
\end{cases}
\end{equation}
If $p=0$ (so that $\gradS$ is non-trivial), we decompose $L^2(M,E\otimes\bF) = L^2(M,E_+\otimes\bF) \oplus L^2(M,E_-\otimes\bF)$ according to the eigenspaces $E_+$ and $E_-$ of $\gradS$ with eigenvalues $+1$ and $-1$, respectively, and with respect to this decomposition we will write 
\[
\pS(\cdot) = \mattwo{0}{\pS_-(\cdot)}{\pS_+(\cdot)}{0} 
\quad\text{and}\quad 
\D = \mattwo{\D|_{E_+}}{0}{0}{\D|_{E_-}} .
\]
Similarly, if $q=0$ (so $\gradD$ is non-trivial), we decompose $L^2(M,E\otimes\bF) = L^2(M,E\otimes\bF_+) \oplus L^2(M,E\otimes\bF_-)$ and write 
\[
\pS(\cdot) = \mattwo{\pS(\cdot)|_{\bF_+}}{0}{0}{\pS(\cdot)|_{\bF_-}} 
\quad\text{and}\quad 
\D = \mattwo{0}{\D_-}{\D_+}{0} .
\]

\begin{defn}
\label{defn:gen_DS}
Consider $M$, $\D$, and $\pS(\cdot)$ satisfying assumption \ref{ass:A}. 
In each of the four possible signatures, we define the \emph{product operator} $\pS \times_{p,q} \D$ on $\til E$ as follows. 
\begin{description}
\item[the odd-odd signature $\boldsymbol{(p,q)=(1,1)}$:] 
We define 
\begin{align*}
\pS \times_{1,1} \D &:= \mattwo{0}{\D+i\pS(\cdot)}{\D-i\pS(\cdot)}{0} ,
\end{align*}
on the initial domain $\big(C_c^1(M,\Dom\pS(x_0)) \otimes_{C_0^1(M)} \Dom\D\big)^{\oplus2}$. 
The operator $\pS \times_{1,1} \D$ is odd with respect to the $\Z_2$-grading $\Gamma := \mattwo{1}{0}{0}{-1}$. 

\item[the even-even signature $\boldsymbol{(p,q)=(0,0)}$:] 
We define 
\begin{align*}
\pS \times_{0,0} \D &:= \D + \pS(\cdot) ,
\end{align*}
on the initial domain $C_c^1(M,\Dom\pS(x_0)) \otimes_{C_0^1(M)} \Dom\D$. 
The operator $\pS \times_{0,0} \D$ is odd with respect to the $\Z_2$-grading $\Gamma := \gradS\otimes\gradD$. 

\item[the odd-even signature $\boldsymbol{(p,q)=(1,0)}$:] 
We define (cf.\ \cite[Example 2.38]{BMS16})
\begin{align*}
\pS \times_{1,0} \D &:= \D + \gradD \pS(\cdot) = \mattwo{\pS(\cdot)|_{\bF_+}}{\D_-}{\D_+}{-\pS(\cdot)|_{\bF_-}} ,
\end{align*}
on the initial domain $C_c^1(M,\Dom\pS(x_0)) \otimes_{C_0^1(M)} \Dom\D$. 
The operator $\pS \times_{1,0} \D$ is ungraded. 

\item[the even-odd signature $\boldsymbol{(p,q)=(0,1)}$:] 
We define (cf.\ \cite[Example 2.37]{BMS16})
\begin{align*}
\pS \times_{0,1} \D &:= \gradS\D + \pS(\cdot) = \mattwo{\D|_{E_+}}{\pS_-(\cdot)}{\pS_+(\cdot)}{-\D|_{E_-}} ,
\end{align*}
on the initial domain $C_c^1(M,\Dom\pS(x_0)) \otimes_{C_0^1(M)} \Dom\D$. 
The operator $\pS \times_{0,1} \D$ is ungraded. 
\end{description}
For arbitrary degrees $(p,q)$, we will often denote the product operator $\pS \times_{p,q} \D$ simply as $\pS \times \D$. 
The product operator is symmetric (hence closable), and (with slight abuse of notation) we will denote the closure simply by $\pS \times \D$ as well. 

The product operator $\pS \times \D$ is called a \emph{Dirac--Schrödinger operator} if $\pS \times \D$ is regular, self-adjoint, and Fredholm. 
In this case, it yields a class (\cref{prop:Fred_KK}) 
\[
[\pS \times \D] \in \KK^{p+q}(\C,A) .
\] 
\end{defn}
We note that, despite our use of the term `Dirac--Schrödinger' operator, we do not assume that the operator $\D$ is of Dirac-type (although a Dirac-type operator is of course a typical example). Furthermore, we note that regularity, self-adjointness, and the Fredholm property of $\pS \times \D$ do not follow automatically from assumption \ref{ass:A}.

\subsection{Relative index theorem}
\label{sec:rel_ind_thm}

The $\K$-theoretic relative index theorem, going back to Bunke \cite{Bun95}, is a `cutting and pasting' tool which is very useful for computing the index (or $\K$-theory class) of a Dirac--Schrödinger operator. 
For $j=1,2$, let $\bF_j\to M_j$, $\D_j$, and $\pS_j(\cdot)$ be as in assumption \ref{ass:A}, 
and assume that the operators $\{\pS_j(x)\}_{x\in M_j}$ act on the same Hilbert $A$-module $E$. 
Suppose we have partitions $M_j = \bar U_j \cup_{N_j} \bar V_j$, where $N_j$ are smooth compact hypersurfaces. 
Let $C_j$ be open tubular neighbourhoods of $N_j$, and assume that there exists an isometry $\phi\colon C_1\to C_2$ (with $\phi(N_1)=N_2$) covered by an isomorphism $\Phi\colon\bF_1|_{C_1} \to \bF_2|_{C_2}$, such that $\D_1|_{C_1} \Phi^* = \Phi^* \D_2|_{C_2}$ and $\pS_2(\phi(x)) = \pS_1(x)$ for all $x\in C_1$. 
We will identify $C_1$ with $C_2$ (as well as $N_1$ with $N_2$) via $\phi$, and we simply write $C$ (and $N$). Define two new Riemannian manifolds 
\begin{align*}
M_3 &:= \bar U_1 \cup_N \bar V_2 , & 
M_4 &:= \bar U_2 \cup_N \bar V_1 .
\end{align*}
Moreover, we glue the bundles using $\Phi$ to obtain hermitian vector bundles $\bF_3\to M_3$ and $\bF_4\to M_4$. For $j=3,4$, we then obtain corresponding operators $\D_j$ and $\pS_j(\cdot)$ satisfying assumption \ref{ass:A}. 
The following statement generalises the relative index theorem of \cite[Theorem 4.1]{vdD25_Callias} to arbitrary signatures $(p,q)$. 
\begin{thm}[Relative index theorem]
\label{thm:rel_index}
Assume that $\pS_1 \times \D_1$ and $\pS_2 \times \D_2$ are Dirac--Schrödinger operators with locally compact resolvents. 
Then $\pS_3 \times \D_3$ and $\pS_4 \times \D_4$ are also Dirac--Schrödinger operators with locally compact resolvents. 
Moreover, we have the equality 
\[
[\pS_1 \times \D_1]+[\pS_2 \times \D_2]=[\pS_3 \times \D_3]+[\pS_4 \times \D_4] \in \KK^{p+q}(\C,A) . 
\]
\end{thm}
\begin{proof}
The proof is very similar to the proofs of \cite[Theorem 4.7]{vdD19_Index_DS} and \cite[Theorem 4.1]{vdD25_Callias}, and therefore we will skip most of the details. 

First, we need to check that $\pS_3 \times \D_3$ and $\pS_4 \times \D_4$ are also regular self-adjoint and Fredholm with locally compact resolvents. This is proven in \cite[Theorem 4.1]{vdD25_Callias} for the signature $(p,q)=(1,1)$. However, in the three other signatures the same proof follows through (taking $\til\D^j:=\D^j$). 

Second, we need to prove that $[\pS_1 \times \D_1]+[\pS_2 \times \D_2]=[\pS_3 \times \D_3]+[\pS_4 \times \D_4]$. 
The proof is given in \cite[Theorem 4.7]{vdD19_Index_DS} for the signature $(p,q)=(1,1)$ (and essentially the same proof works for $(p,q)=(0,0)$). 
We will adapt the proof such that it works in all signatures $(p,q)$. 
For $j=1,\ldots,4$, we write $\til E_j := L^2(M_j,E\otimes\bF_j)^{\oplus2}$ if $(p,q)=(1,1)$ and $\til E_j := L^2(M_j,E\otimes\bF_j)$ otherwise (as in \cref{eq:tilE}). 
If $p+q\equiv0$, each $\til E_j$ is $\Z_2$-graded with grading operator $\Gamma_j$. 
We consider the Hilbert $A$-module $\til E_1\oplus \til E_2\oplus \til E_3\oplus \til E_4$, which we equip with the grading operator $\Gamma_1\oplus\Gamma_2\oplus(-\Gamma_3)\oplus(-\Gamma_4)$ if $p+q\equiv0$. 
We consider the operator 
\[
T := \big( \pS_1 \times \D_1 \big) \oplus \big( \pS_2 \times \D_2 \big) \oplus - \big( \pS_3 \times \D_3 \big) \oplus {-\big(\pS_4 \times \D_4 \big)} . 
\]
Since $[T] = \big[ \pS_1 \times \D_1 \big] + \big[ \pS_2 \times \D_2 \big] - \big[ \pS_3 \times \D_3 \big] - \big[ \pS_4 \times \D_4 \big] \in \KK^{p+q}(\C,A)$, we need to show that $[T]=0$. 
With the maps $\alpha_{ij}$ defined as in \cite[Proof of Theorem 4.7]{vdD19_Index_DS}, we now define (instead of $X$) the operator 
\[
\gamma := \matfour{0&0&\alpha_{13}^*&\alpha_{14}^*}{0&0&\alpha_{23}^*&-\alpha_{24}^*}{\alpha_{13}&\alpha_{23}&0&0}{\alpha_{14}&-\alpha_{24}&0&0} .
\]
Then $\gamma=\gamma^*$, $\gamma^2=1$, and if $p+q\equiv0$, $\gamma$ is odd. 
Moreover, we compute 
\[
T \gamma + \gamma T 
= \matfour{0&0&[\til\D_1,\chi_{U_1}]&[\til\D_1,\chi_{V_1}]}{0&0&[\til\D_2,\chi_{V_2}]&[-\til\D_2,\chi_{U_2}]}{{-[\til\D_1,\chi_{U_1}]}&-[\til\D_2,\chi_{V_2}]&0&0}{{-[\til\D_1,\chi_{V_1}]}&[\til\D_2,\chi_{U_2}]&0&0} ,
\]
where we write $\til\D_j := \mattwo{0}{\D_j}{\D_j}{0}$ if $(p,q)=(1,1)$ and $\til\D_j := \D_j$ otherwise. 
Since $T \gamma + \gamma T$ is relatively $T$-compact, it then follows from \cref{prop:Fred_almost-Cliff-symm} that $[T] = 0 \in \KK^{p+q}(\C,A)$. 
\end{proof}

\subsection{The Fredholm index}
\label{sec:gen_DS_Fred}

We denote by $[\cdot,\cdot]_\pm$ the graded commutator. Thus $[\D,\pS(\cdot)]_\pm = [\D,\pS(\cdot)]_+$ is the anticommutator for $(p,q)=(0,0)$ and $[\D,\pS(\cdot)]_\pm = [\D,\pS(\cdot)]_-$ is the commutator otherwise. 
From here on, we consider in addition to assumption \ref{ass:A} also the following assumption:
\begin{assumption*}[(B)]
\customlabel{ass:B}{(B)}
Let $M$, $\D$, and $\pS(\cdot)$ satisfy assumption \ref{ass:A}. 
We furthermore assume:
\begin{itemize}
\item[(B1)]
\customlabel{ass:B1}{(B1)}
the map $\pS(\cdot) \colon M \to \mL_A(W,E)$ is weakly differentiable (i.e., for each $\psi\in W$ and $\eta\in E$, the map $x \mapsto \la\pS(x)\psi|\eta\ra$ is differentiable), and the weak derivative $d\pS(x) \colon W \to E\otimes T_x^*(M)$ is bounded for all $x\in M$. 
\item[(B2)]
\customlabel{ass:B2}{(B2)} 
the operator $\big[\D,\pS(\cdot)\big]_\pm \big(\pS(\cdot)\pm i\big)^{-1}$ is well-defined and bounded (in the sense of \cite[Assumption 7.1]{KL12} and \cite[Definition 5.5]{vdD19_Index_DS}): 
there exists a core $\E\subset\Dom\D$ for $\D$ such that for all $\xi\in\E$ and for all $\mu\in(0,\infty)$ we have the inclusions 
\[
\big(\pS(\cdot)\pm i\mu\big)^{-1} \xi \in \Dom\pS(\cdot) \cap \Dom\D 
\quad \text{and} \quad 
\D \big(\pS(\cdot)\pm i\mu\big)^{-1} \xi \in \Dom\pS(\cdot) ,
\]
and the map $\big[\D,\pS(\cdot)\big]_\pm \big(\pS(\cdot)\pm i\mu\big)^{-1} \colon \E \to L^2(M,E\otimes\bF)$ extends to a bounded operator for all $\mu\in(0,\infty)$. 
\end{itemize}
\end{assumption*}

\begin{prop}[cf.\ {\cite[Theorem 7.10]{KL12}}]
\label{prop:reg-sa}
The product operator $\pS \times_{p,q} \D$ is regular self-adjoint on the domain $\Dom\big(\pS \times_{1,1} \D\big) = \big(\Dom\pS(\cdot)\cap\Dom\D\big)^{\oplus2}$ if $(p,q)=(1,1)$ or $\Dom\big(\pS \times_{p,q} \D\big) = \Dom\pS(\cdot)\cap\Dom\D$ otherwise. 
\end{prop}
\begin{proof}
The odd-odd signature $(p,q)=(1,1)$ is precisely the statement of \cite[Theorem 7.10]{KL12}, which relies solely on the well-definedness and boundedness of $\big[\D,\pS(\cdot)\big]_- \big(\pS(\cdot)\pm i\big)^{-1}$ as given by assumption \ref{ass:B2}. 
If instead $\big[\D,\pS(\cdot)\big]_+ \big(\pS(\cdot)\pm i\big)^{-1}$ is well-defined and bounded, as in the even-even signature $(p,q)=(0,0)$, then the regularity and self-adjointness of $\pS \times_{0,0} \D = \D + \pS(\cdot)$ follows from a standard matrix trick (see e.g.\ \cite[\S2.2]{vdDR16} for details). 
Furthermore, for the odd-even signature $(p,q)=(1,0)$ resp.\ the even-odd signature $(p,q)=(0,1)$, 
we note that $[\D,\gradD\pS(\cdot)]_+ = -\gradD [\D,\pS(\cdot)]_-$ resp.\ $[\gradS\D,\pS(\cdot)]_+ = \gradS [\D,\pS(\cdot)]_-$. 
Using assumption \ref{ass:B2}, the regularity and self-adjointness of $\pS \times_{1,0} \D = \D + \gradD \pS(\cdot)$ resp.\ $\pS \times_{0,1} \D = \gradS\D + \pS(\cdot)$ then follow as in the case $(p,q)=(0,0)$. 
\end{proof}

\begin{prop}[cf.\ {\cite[Theorem 6.7]{KL13}}]
\label{prop:cpt_res}
Let $\phi\in C_0(M)$. 
Then $\phi(\pS\times\D \pm i)^{-1}$ is a compact operator on $\til E$. 
Moreover, if $(\pS(\cdot)\pm i)^{-1}$ is compact on $C_0(M,E)$, then $(\pS\times\D \pm i)^{-1}$ is also compact. 
\end{prop}
\begin{proof}
The statement is proven in \cite[Theorem 6.7]{KL13} for the signature $(p,q)=(1,1)$, but in fact the same proof also works for the other signatures. 
Moreover, if $(\pS(\cdot)\pm i)^{-1}$ is in fact compact, then we can repeat the argument with $\phi=1$ to conclude that also $(\pS\times\D \pm i)^{-1}$ is compact. 
\end{proof}

\begin{thm}[cf.\ {\cite[Theorem 4.4]{vdD25_Callias}}]
\label{thm:Fredholm}
\mbox{}
\begin{enumerate}
\item 
There exists $\lambda_0>0$ such that for any $\lambda\geq\lambda_0$ the product operator $(\lambda\pS)\times\D$ is Fredholm and thus a Dirac--Schrö\-ding\-er operator. 
\item \label{item:thm:Fredholm2}
Suppose there exists a compact subset $\hat K \supset K$ such that $\hat\delta < \frac{\hat c^2}{\hat c+1}$, where 
\begin{align*}
\hat\delta &:= \sup_{x\in M\setminus\hat K} \big\| \big[\D,\pS(\cdot)\big]_\pm(x) \big(\pS(x)\pm i\big)^{-1} \big\| , &
\hat c &:= \inf_{x\in M\setminus\hat K} \|\pS(x)^{-1}\|^{-1} .
\end{align*}
Then the first statement holds with $\lambda_0=1$. 
In particular, $\pS\times\D$ is Fredholm and thus a Dirac--Schrödinger operator. 
\end{enumerate}
\end{thm}
\begin{proof}
The signature $(p,q)=(1,1)$ is given by \cite[Theorem 4.4]{vdD25_Callias}. 
However, the same proof also works for the other signatures, with the following notational modifications: we set $\til\D := \D$ or $\til\D := \gradS\D$ and $\til\pS(\cdot) := \pS(\cdot)$ or $\til\pS(\cdot) := \gradD \pS(\cdot)$ (depending on the signature), $\til\D + \lambda\til\pS(\cdot) = (\lambda\pS)\times\D$, and we define $\delta_x := \big\| \big[\D,\pS(\cdot)\big]_\pm(x) \big(\pS(x)\pm i\big)^{-1} \big\|$. 
\end{proof}

\begin{prop}[cf.\ {\cite[Proposition 4.5]{vdD25_Callias}}]
\label{prop:D_S_invertible}
Suppose that $\{\pS(x)\}_{x\in M}$ is uniformly invertible on \emph{all} of $M$. 
Then there exists $\lambda_0>0$ such that for any $\lambda\geq\lambda_0$ the Dirac--Schrödinger operator $(\lambda\pS)\times\D$ is also invertible. 
\end{prop}

\cref{thm:Fredholm} shows the Fredholm property of $(\lambda\pS)\times\D$ only for $\lambda$ `sufficiently large'; under suitable conditions, we can take $\lambda=1$. 
The following result provides sufficient conditions such that we can take any $\lambda>0$. 
\begin{prop}
\label{prop:Fredholm_all_lambda}
Suppose that the operator $\big[\D,\pS(\cdot)\big]_\pm \big(\pS(\cdot)\pm i\big)^{-1}$ vanishes at infinity, (in the sense that the function $M\to\R$, $x \mapsto \big\| \big[\D,\pS(\cdot)\big]_\pm(x) \big(\pS(x)\pm i\big)^{-1} \big\|$ vanishes at infinity). 
Then for any $\lambda>0$, the product operator $(\lambda\pS)\times\D$ is Fredholm. 
\end{prop}
\begin{proof}
Fix $\lambda>0$. We note that the family $\lambda\pS(\cdot)$ also satisfies assumptions \ref{ass:A} and \ref{ass:B}. 
Let $c_\lambda := \inf_{x\in M\setminus K} \|\lambda^{-1} \pS(x)^{-1}\|^{-1}$. 
By assumption, there exists a compact subset $\hat K \supset K$ such that for all $x\in M\setminus\hat K$ we have $\big\| \big[\D,\pS(\cdot)\big]_\pm(x) \big(\pS(x)\pm i\big)^{-1} \big\| \leq \min(\lambda^{-1},1) \frac{c_\lambda^2}{c_\lambda+1}$. 
We can estimate 
\[
\left\| \big(\pS(x)\pm i\big) \big(\lambda\pS(x)\pm i\big)^{-1} \right\| \leq \sup_{s\in\R} \left| \frac{s\pm i}{\lambda s\pm i} \right| = \max(1,\lambda^{-1}) .
\]
Consequently, 
\begin{align*}
\hat \delta_\lambda 
&:= \sup_{x\in M\setminus\hat K} \big\| \big[\D,\lambda\pS(\cdot)\big]_\pm(x) \big(\lambda\pS(x)\pm i\big)^{-1} \big\| \\
&\leq \lambda \cdot \max(1,\lambda^{-1}) \cdot \sup_{x\in M\setminus\hat K} \big\| \big[\D,\pS(\cdot)\big]_\pm(x) \big(\pS(x)\pm i\big)^{-1} \big\| 
\leq \frac{c_\lambda^2}{c_\lambda+1} .
\end{align*}
Since $K\subset\hat K$, we have $c_\lambda \leq \hat c_\lambda := \inf_{x\in M\setminus\hat K} \|\lambda^{-1} \pS(x)^{-1}\|^{-1}$, and therefore also 
\[
\hat \delta_\lambda \leq \frac{c_\lambda^2}{c_\lambda+1} \leq \frac{\hat c_\lambda^2}{\hat c_\lambda+1} .
\]
It then follows from \cref{thm:Fredholm}.\ref{item:thm:Fredholm2} that $(\lambda\pS)\times\D$ is Fredholm. 
\end{proof}

\subsubsection{A vanishing result}

Let $M$, $\D$ and $\pS$ satisfy assumptions \ref{ass:A} and \ref{ass:B} with signature $(p,q)$, such that $(\lambda\pS)\times\D$ is Fredholm (and hence a Dirac--Schrödinger operator) for $\lambda \geq \lambda_0 > 0$. 
In the case where $\D$ is ungraded ($q=1$) and $\D$ has a gap in the spectrum, we obtain the following vanishing result for the $\K$-theory class of the Dirac--Schrödinger operator. This vanishing was observed by Bunke \cite[Proposition 2.8]{Bun95} in the odd-odd signature $(p,q)=(1,1)$, but also holds for the even-odd signature $(p,q)=(0,1)$. 

\begin{prop}
\label{prop:D-gap}
Let $(\lambda\pS)\times\D$ be a Dirac--Schrödinger operator with signature $(p,1)$, such that the following conditions hold:
\begin{itemize}
\item $\pS \colon M \to \mL_A(E)$ is bounded; 
\item $\D$ has a gap in the spectrum; and 
\item the operator $\big[\D,\pS(\cdot)\big] \big(\pS(\cdot)\pm i\big)^{-1}$ vanishes at infinity. 
\end{itemize}
Then $[\pS\times\D] = 0 \in \KK^{q+1}(\C,A)$. 
\end{prop}
\begin{proof}
Let $\mu\in\R$ such that $\D-\mu$ is invertible. 
For $s,t\in[0,1]$ consider the Dirac--Schrödinger operator $(t\pS)\times(\D-s\mu)$ satisfying the assumptions \ref{ass:A} and \ref{ass:B}. 
Since $\pS$ and $\mu$ are bounded, $(t,s) \mapsto (t\pS)\times(\D-s\mu)$ yields a continuous family of operators $\Dom\D \to \til E$. 
By assumption, we know for all $s$ that $\big[\D-s\mu,\pS(\cdot)\big] \big(\pS(\cdot)\pm i\big)^{-1}$ vanishes at infinity. 
By \cref{prop:Fredholm_all_lambda}, $(t\pS)\times(\D-s\mu)$ is Fredholm for all $t>0$ (and for all $s$). 
Furthermore, the operator $0\times(\D-\mu)$ is invertible. 
Hence $\pS\times\D$ is homotopic (within the space of Fredholm operators with domain $\Dom\D$) to $\pS\times(\D-\mu)$, which in turn is homotopic to the invertible operator $0\times(\D-\mu)$. 
Thus we conclude that $[\pS\times\D] = 0$. 
\end{proof}

\subsection{The pairing}
\label{sec:gen_DS_Kasp_prod}

We will show next that a Dirac--Schrödinger operator $\pS\times\D$ represents the pairing of the $\K$-theory class of $\pS$ with the $\K$-homology class of $\D$. 
As a first step, following \cite[Proposition 4.6]{vdD25_Callias}, we can apply the relative index theorem (\cref{thm:rel_index}) to replace $M$ by a manifold with cylindrical ends, without affecting the $\K$-theory class of the Dirac--Schrödinger operator. 
\begin{prop}[cf.\ {\cite[Proposition 4.6]{vdD25_Callias}}]
\label{prop:cyl_ends}
There exist a precompact open subset $U$ of $M$ and a Dirac--Schrödinger operator $(\lambda\pS')\times\D'$ on $M' := \bar U \cup_{\partial U} (\partial U \times [0,\infty))$ satisfying Assumptions \ref{ass:A} and \ref{ass:B}, such that 
\begin{enumerate}
\item 
the operators $\D'$ and $\pS'(\cdot)$ on $M'$ agree with $\D$ and $\pS(\cdot)$ on $M$ when restricted to $U$; 
\item 
the metric and the operators $\D'$ and $\pS'(\cdot)$ on $M'$ are of product form on $\partial U \times [1,\infty)$; 
\item 
we have, for $\lambda$ sufficiently large, the equality $[(\lambda\pS')\times\D'] = [(\lambda\pS)\times\D] \in \KK^{p+q}(\C,A)$. 
\end{enumerate}
In particular, $M'$ is complete and $\D'$ has bounded propagation speed. 
\end{prop}

\begin{prop}[cf.\ {\cite[Proposition 4.7]{vdD25_Callias}}]
\label{prop:cyl_ends_Kasp_prod}
Let $M'$, $\D'$, and $\pS'(\cdot)$ be as in \cref{prop:cyl_ends}. 
Then we have the equality $[\pS'(\cdot)] \otimes_{C_0(M')} [\D'] = [(\lambda\pS')\times\D']$. 
\end{prop}
\begin{proof}
This follows from the same arguments in \cite[Proposition 4.7]{vdD25_Callias} and \cite[Proposition 5.10]{vdD19_Index_DS} (which prove the case $(p,q)=(1,1)$), noting that also in the other signatures we have defined the product operator in \cref{defn:gen_DS} in such a way that it correctly corresponds to the construction of the Kasparov product as in \cite[2.34--2.39]{BMS16}. 
\end{proof}

\begin{thm}[cf.\ {\cite[Theorem 3.5]{vdD25_Callias}}]
\label{thm:Kasp_prod_index}
Let $M$ be a connected Riemannian manifold, and let $\{\pS(x)\}_{x\in M}$ and $\D$ satisfy assumptions \ref{ass:A} and \ref{ass:B}. 
Then there exists $\lambda_0>0$ such that for any $\lambda\geq\lambda_0$ the class $[(\lambda\pS)\times\D] \in \KK^{p+q}(\C,A)$ is the (internal) Kasparov product (over $C_0(M)$) of $[\pS(\cdot)] \in \KK^p(\C,C_0(M,A))$ with $[\D] \in \KK^q(C_0(M),\C)$. 
\end{thm}
\begin{proof}
The case $(p,q)=(1,1)$ is given by \cite[Theorem 3.5]{vdD25_Callias}, but the same proof works for arbitrary $(p,q)$. 
Indeed, from \cref{prop:cyl_ends}, we obtain a complete manifold $M'$ and a Dirac--Schrödinger operator satisfying assumptions \ref{ass:A} and \ref{ass:B}, such that $\D'$ has bounded propagation speed, and such that $[(\lambda\pS')\times\D'] = [(\lambda\pS)\times\D]$ (for $\lambda$ sufficiently large). 
We then have the equalities 
\[
[(\lambda\pS)\times\D] 
= [(\lambda\pS')\times\D'] 
= [\pS'(\cdot)] \otimes_{C_0(M')} [\D'] 
= [\pS(\cdot)] \otimes_{C_0(M)} [\D] ,
\]
where the second equality is given by \cref{prop:cyl_ends_Kasp_prod}, 
and the third equality is obtained as in the proof of \cite[Theorem 5.15]{vdD19_Index_DS}. 
\end{proof}

\begin{remark}[Equivariant spectral flow] 
The presence of the auxiliary $C^*$-algebra $A$ in the above theorem allows us to immediately obtain also an `equivariant version', where the potential $\pS(\cdot)$ is obtained from a family of $G$-equivariant operators. As an application, we briefly describe here the equivariant spectral flow (as studied for instance in \cite{HY25}). 

Consider a complete Riemannian manifold $X$ and a unimodular, locally compact group $G$ acting cocompactly, properly, and isometrically on $X$. 
Let $\pS$ be a $G$-equivariant, self-adjoint, elliptic, first-order differential operator on a $G$-equivariant hermitian vector bundle over $X$ (where $G$ acts via isometries between the fibres). 
Kasparov \cite{Kas16} has then constructed a Hilbert $C^*(G)$-module $\E^0(E)$. 
The operator $\pS$ can then be viewed as a regular self-adjoint operator on $\E^0(E)$ \cite[Proposition 5.5]{Guo21}. 

Now, instead of a single operator $\pS$, we consider a family of operators $\pS(\cdot) = \{\pS(x)\}_{x\in M}$ parametrised by another manifold $M$. If $M$, $\D$, and $\pS(\cdot)$ satisfy assumptions \ref{ass:A} and \ref{ass:B}, then it follows from \cref{thm:Kasp_prod_index} (with the signature $(p,q)=(1,1)$) that
\[
\Index\big( \D - i \lambda\pS(\cdot) \big) = [(\lambda\pS)\times_{1,1}\D] = [\pS(\cdot)] \otimes_{C_0(M)} [\D] \in \K_0\big( C^*(G) \big) .
\]
This recovers (a slight variation of) a recent result by Hochs and Yanes \cite[Theorem 3.38]{HY25}. 
In the special case $M=\R$ and $\D=-i\partial_x$, 
provided there exist locally trivialising families for $\{\pS(x)\}_{x\in\R}$, 
we obtain from \cref{prop:SF_Kasp_prod}
an equivariant version of the classical `index = spectral flow' equality (see also \cite[Corollary 5.16]{vdD19_Index_DS}): 
\[
\Index\big( \partial_x + \lambda\pS(\cdot) \big) = \SF_0\big(\{\pS(x)\}_{x\in\R}\big) \in \K_0\big( C^*(G) \big) .
\]
\end{remark}

\section{Callias-type operators}
\label{sec:gen_Callias}

The classical Callias Theorem, first proven by Callias \cite{Cal78} on Euclidean space, and subsequently generalised by various authors (see, for instance, \cite{Ang90,BM92,Ang93a,Rad94,Bun95,Kuc01,Gesztesy-Waurick16}), states that the index of a Dirac-Schrödinger operator $\D-i\pS$ (with signature $(p,q)=(1,1)$) can be computed on a suitable compact hypersurface. 
An analogous statement for the even-even signature $(p,q)=(0,0)$ has been described by Bunke \cite[\S2.4]{Bun95}. 
A generalisation for potentials consisting of families of \emph{unbounded} operators (as in Assumption \ref{ass:A}) is given in \cite[Theorem 3.8]{vdD25_Callias} for the signature $(p,q)=(1,1)$. 
In this section, we provide a general Callias Theorem for arbitrary signatures $(p,q)$. 

\begin{defn}
\label{defn:D_product_form}
Let $M$, $\bF$, and $\D$ be as in assumption \ref{ass:A}. 
Let $N\subset M$ be a compact hypersurface with a collar neighbourhood $C \simeq (-2\varepsilon,2\varepsilon) \times N$. 
We can identify $\bF|_C$ with the pullback of $\bF|_N \to N$ to $C$, so that $\Gamma^\infty(\bF|_C) \simeq C^\infty\big( (-2\varepsilon,2\varepsilon) \big) \otimes \Gamma^\infty(\bF|_N)$. 
We distinguish between the graded case $q=0$ (with $\bF$ $\Z_2$-graded and $\D$ odd) and the ungraded case $q=1$. 
We will then say that \emph{$\D$ is of product form on $C$} if the following holds:
\begin{description}
\item[$\boldsymbol{q=0}$:] 
We have a vector bundle $\bF_N\to N$ such that $\bF|_N \simeq \bF_N \oplus \bF_N$. 
On $\bF_N$ we have an essentially self-adjoint elliptic first-order differential operator $\D_N$. 
On the collar neighbourhood $C$, we have 
\[
\D|_C \simeq \mattwo{0}{\partial_r+\D_N}{-\partial_r+\D_N}{0} ,
\qquad
\gradD\big|_C \simeq \mattwo{1}{0}{0}{-1} .
\]
\item[$\boldsymbol{q=1}$:] 
On the vector bundle $\bF_N:=\bF|_N\to N$, we have an essentially self-adjoint elliptic first-order differential operator $\D_N$ and a self-adjoint unitary $\Gamma_N \in \Gamma^\infty(\End\bF_N)$ satisfying $\Gamma_N \D_N = - \D_N \Gamma_N$. 
On the collar neighbourhood $C$, we have 
\[
\D|_C \simeq -i\partial_r \otimes \Gamma_N + 1 \otimes \D_N . 
\]
\end{description}
\end{defn}

Throughout this section, we will assume the following: 
\begin{assumption*}[(C)]
\customlabel{ass:C}{(C)}
Let $M$, $\D$ and $\pS(\cdot)$ satisfy assumptions \ref{ass:A} and \ref{ass:B} such that $(\lambda\pS)\times\D$ is Fredholm (and hence a Dirac--Schrödinger operator) for $\lambda \geq \lambda_0 > 0$. 
Without loss of generality, assume that the compact subset $K$ from Assumption \ref{ass:A3} has a smooth compact boundary $N$. 
We assume furthermore that the following conditions are satisfied:
\begin{itemize}
\item[(C1)]
\customlabel{ass:C1}{(C1)}
The operator $\D$ is of product form on a collar neighbourhood $C \simeq (-2\varepsilon,2\varepsilon) \times N$ of $N$ (with $(-2\varepsilon,0) \times N$ in the interior of $K$). 
\item[(C2)]
\customlabel{ass:C2}{(C2)}
For any $x,y\in K$, $\pS(x)-\pS(y)$ is relatively $\pS(x)$-compact. 
\end{itemize}
Moreover, we fix an (arbitrary) invertible regular self-adjoint operator $\pT$ on $E$ with domain $\Dom\pT = W$, such that $\pS(x)-\pT$ is relatively $\pT$-compact for some (and hence, by \ref{ass:C2}, for every) $x\in K$. 
If $p=0$, $\pT$ is also assumed to anti-commute with $\gradS$. 
\end{assumption*}

\begin{defn}
If assumptions \ref{ass:A}, \ref{ass:B}, and \ref{ass:C} are satisfied, then the Dirac--Schrödinger operator $(\lambda\pS)\times\D$ is called a \emph{Callias-type operator}. 

\noindent 
(We always implicitly assume that $\lambda\geq\lambda_0>0$ such that $(\lambda\pS)\times\D$ is Fredholm.)
\end{defn}

We consider the invertible regular self-adjoint operator $\pT(\cdot)$ on $C(N,E)$ corresponding to the constant family $\pT(y) := \pT$ (for $y\in N$). 
The restriction of the potential $\pS(\cdot)$ to the hypersurface $N$ also yields an invertible regular self-adjoint operator $\pS_N(\cdot) = \{\pS(y)\}_{y\in N}$ on $C(N,E)$. 
We recall that $\pS(y) - \pT$ is relatively $\pT$-compact for each $y\in N$. 
Furthermore, $\pS(y) \big( \pT\pm i \big)^{-1}$ depends norm-continuously on $y$ by Assumption \ref{ass:A2}. 
Hence $\pS_N(\cdot) - \pT(\cdot)$ is relatively $\pT(\cdot)$-compact. 
We then know from \cite[Corollary A.10]{vdD25_Callias} that $P_+(\pS_N(\cdot)) - P_+(\pT(\cdot))$ is compact. 
Note that for $p=0$, $\pT(\cdot)$ and $\pS_N(\cdot)$ are graded, whereas for $p=1$, $\pT(\cdot)$ and $\pS_N(\cdot)$ are ungraded. 
In both cases, the (odd or even) relative index $\relind_{p+1}\big(P_+(\pS_N(\cdot)),P_+(\pT(\cdot))\big)$ is well-defined (see \cref{defn:rel-ind}). 
We are now ready to state our generalised Callias Theorem for arbitrary signatures $(p,q)$. 

\begin{thm}[Callias Theorem]
\label{thm:Callias}
Let $(\lambda\pS)\times\D$ be a Callias-type operator. 
Then we have the equality
\begin{align*}
[(\lambda\pS)\times\D]
= \relind_{p+1}\big(P_+(\pS_N(\cdot)),P_+(\pT(\cdot))\big) \otimes_{C(N)} [\D_N] \in \K_{p+q}(A) .
\end{align*}
\end{thm}

Before proceeding with the proof of the theorem, we observe that our assumption \ref{ass:C2} ensures that the class $[\pS(\cdot)]$ of the potential depends only on the hypersurface $N$. 
Consider the open subset $U := K\cup C \subset M$ with compact closure $\bar U$ and boundary $\partial U\simeq N$. 
We have the short exact sequence 
\begin{align}
\label{eq:ses_U}
0 &\rightarrow C_0(U,A) \to C(\bar U,A) \to C(N,A) \to 0 
\end{align}
and the corresponding cyclic six-term exact sequences in $\K$-theory and $\K$-homology. 

\begin{prop}
\label{prop:potential_boundary}
The class $[\pS(\cdot)] \in \K_p(C_0(M,A))$ is uniquely determined by an element $\Sigma_N \in \K_{p+1}(C(N,A))$. 
More explicitly, we have 
\[
[\pS(\cdot)] = {\iota_U}_* \circ \partial(\Sigma_N) ,
\]
where $\partial \colon \K_{p+1}(C(N,A)) \to \K_p(C_0(U,A))$ denotes the boundary map in the cyclic six-term exact sequence in $\K$-theory corresponding to the short exact sequence \eqref{eq:ses_U}, 
and where ${\iota_U}_* \colon \K_p(C_0(U,A)) \to \K_p(C_0(M,A))$ is induced by the inclusion $\iota_U \colon C_0(U,A) \into C_0(M,A)$. 
\end{prop}
\begin{proof}
The proof given in \cite[Proposition 3.10]{vdD25_Callias} for the signature $(p,q)=(1,1)$ also works for the other signatures (replacing $\K_1$ by $\K_p$ and $\K_0$ by $\K_{p+1}$). 
\end{proof}

As explained in \cite[\S3.2]{vdD25_Callias}, the above proposition already implies the equality 
\[
[(\lambda\pS)\times\D]
= \Sigma_N \otimes_{C(N)} [\D_N] \in \K_{p+q}(A) .
\]
In order to prove \cref{thm:Callias}, it only remains to identify an explicit choice for the $\K$-theory class $\Sigma_N \in \K_0(C(N,A))$ (note that the boundary map is not necessarily injective, so that $\Sigma_N$ is not uniquely determined). 
For this final step, we adapt the approach of \cite[\S5]{vdD25_Callias} to arbitrary signatures, and we will show that $\Sigma_N$ may be taken to be the relative index of the positive spectral projections $P_+(\pS_N(\cdot))$ and $P_+(\pT(\cdot))$.

\subsection{Proof of the theorem}

Let $(\lambda\pS)\times\D$ be a Callias-type operator. 
We will show that we can replace the manifold $M$ by a cylindrical manifold $\R\times N$, without changing the $\K$-theory class of $(\lambda\pS)\times\D$. 
Thus we can reduce the proof of our Callias Theorem from the general statement to the case of a cylindrical manifold. 
This reduction is made possible by the relative index theorem (\cref{thm:rel_index}). 

\begin{lem}
\label{lem:collar_potential}
We may replace the collar neighbourhood $C$ by a smaller collar neighbourhood $C' \simeq (-2\varepsilon',2\varepsilon')\times N$ (with $0 < \varepsilon' < \varepsilon$) and the potential $\pS(\cdot)$ by a potential $\pS'(\cdot)$ satisfying:
\begin{itemize}
\item 
for all $x\in K\setminus C'$: $\pS'(x) = \pT$;
\item 
for all $x=(r,y) \in C'$: $\pS'(x) = \varrho(r) \pT + \big(1-\varrho(r)\big) \pS(y)$, 
for some function $\varrho \in C^\infty(\R)$ such that $0 \leq \varrho(r) \leq 1$ for all $r\in\R$, $\varrho(r)=1$ for all $r$ in a neighbourhood of $(-\infty,-\varepsilon']$, and $\varrho(r)=0$ for all $r$ in a neighbourhood of $[0,\infty)$, 
\end{itemize}
such that $[\pS(\cdot)] = [\pS'(\cdot)] \in \K_p(C_0(M,A))$ and (for $\lambda$ sufficiently large) $[(\lambda\pS)\times\D] = [(\lambda\pS')\times\D] \in \K_{p+q}(A)$. 
\end{lem}
\begin{proof}
The proof given in \cite[Lemma 5.1]{vdD25_Callias} for the signature $(p,q)=(1,1)$ also works for the other signatures, using \cref{thm:Kasp_prod_index} instead of \cite[Theorem 3.5]{vdD25_Callias}. 
\end{proof}

\begin{defn}
\label{defn:cylinder}
Consider the cylindrical manifold $\R\times N$, along with the pullback vector bundle $\bF_{\R\times N}$ obtained from $\bF|_N\to N$. 
We identify $\Gamma_c^\infty(\bF_{\R\times N}) \simeq C_c^\infty(\R) \otimes \Gamma^\infty(\bF|_N)$, 
and consider the essentially self-adjoint elliptic first-order differential operator $\D_{\R\times N}$ on $\bF_{\R\times N}$ given by 
\begin{equation*}
\D_{\R\times N} := 
\begin{cases}
\mattwo{0}{\partial_r+\D_N}{-\partial_r+\D_N}{0} , \quad& \text{if } q=0 , \\
-i\partial_r \otimes \Gamma_N + 1 \otimes \D_N , \quad& \text{if } q=1 .
\end{cases}
\end{equation*}
Let $\varrho \in C^\infty(\R)$ be as in \cref{lem:collar_potential}. 
We define the family $\{\pS_{\R\times N}(r,y)\}_{(r,y)\in\R\times N}$ on $E$ given by
\begin{align*}
\pS_{\R\times N}(r,y) := \varrho(r) \pT + \big(1-\varrho(r)\big) \pS(y) .
\end{align*}
\end{defn}

If $q=0$, the operator $\D_{\R\times N}$ has the same $\Z_2$-grading as in \cref{defn:D_product_form}, and therefore yields an even $\K$-homology class $[\D_{\R\times N}] \in \KK^0(C_0(\R\times N),\C)$. In this case, $\D_N$ is ungraded and defines an odd $\K$-homology class $[\D_N] \in \KK^{1}(C(N),\C)$. 

If $q=1$, the ungraded operator $\D_{\R\times N}$ yields an odd $\K$-homology class $[\D_{\R\times N}] \in \KK^1(C_0(\R\times N),\C)$. 
The operator $\Gamma_N$ from \cref{defn:D_product_form} provides a $\Z_2$-grading on $\bF_N$, yielding the decomposition $\bF_N = \bF_N^+ \oplus \bF_N^-$. By assumption, $\D_N$ is odd with respect to this $\Z_2$-grading, and thus $\D_N$ defines an even $\K$-homology class $[\D_N] \in \KK^{0}(C(N),\C)$. 

\begin{lem}
\label{lem:D_product}
The external Kasparov product of $[-i\partial_r] \in \KK^1(C_0(\R),\C)$ with $[\D_N] \in \KK^{q+1}(C(N),\C)$ equals $[\D_{\R\times N}] \in \KK^q(C_0(\R\times N),\C)$. 
\end{lem}
\begin{proof}
The statement follows from the description of the odd-even resp.\ odd-odd (internal) Kasparov product given in \cite[Example 2.38]{BMS16} resp.\ \cite[Example 2.39]{BMS16} (up to unitary isomorphism, and noting that the argument remains valid in the simpler case of an external Kasparov product). 
\end{proof}

\begin{thm}
\label{thm:red_cyl}
Consider the cylindrical manifold $\R\times N$ with the operators $\D_{\R\times N}$ and $\pS_{\R\times N}(\cdot)$ from \cref{defn:cylinder}. 
Then, for $\lambda$ sufficiently large, 
\[
\big[ (\lambda\pS)\times\D \big] = \big[ (\lambda\pS_{\R\times N})\times\D_{\R\times N} \big] . 
\]
\end{thm}
\begin{proof}
The proof given in \cite[Theorem 5.4]{vdD25_Callias} for the signature $(p,q)=(1,1)$ also works for the other signatures, using the relative index theorem for arbitrary signatures (\cref{thm:rel_index}) instead of \cite[Theorem 4.1]{vdD25_Callias}. 
\end{proof}

We are now ready to prove our Callias Theorem. 

\begin{proof}[\textbf{Proof of \cref{thm:Callias}}]
The proof is similar to the proof of \cite[Theorem 3.8]{vdD25_Callias}, which we adapt to all four possible signatures $(p,q)$. 
Consider the cylindrical manifold $\R\times N$ with the operators $\D_{\R\times N}$ and $\pS_{\R\times N}(\cdot)$ from \cref{defn:cylinder}. 
From \cref{prop:sf_rel_cpt_family,prop:SF_Kasp_prod} we have 
\begin{equation*}
\relind_{p+1}\big(P_+(\pS_N(\cdot)),P_+(\pT(\cdot))\big)
= \SF_{p+1} \big( \{\pS_{\R\times N}(r)\}_{r\in[-\epsilon,0]} \big) 
= [\pS_{\R\times N}(\cdot)] \otimes_{C_0(\R)} [-i\partial_r] . 
\end{equation*}
We then compute 
\begin{align*}
\Big( [\pS_{\R\times N}(\cdot)] \otimes_{C_0(\R)} [-i\partial_r] \Big) \otimes_{C(N)} [\D_N] 
&= [\pS_{\R\times N}(\cdot)] \otimes_{C_0(\R\times N)} \Big( [-i\partial_r] \otimes [\D_N] \Big) \\
&= [\pS_{\R\times N}(\cdot)] \otimes_{C_0(\R\times N)} [\D_{\R\times N}] ,
\end{align*}
where the first equality follows from the properties of the Kasparov product, and the second equality is given by \cref{lem:D_product}. 
Furthermore, since the operators $\pS_{\R\times N}(\cdot)$ and $\D_{\R\times N}$ on the manifold $\R\times N$ satisfy the assumptions \ref{ass:A} and \ref{ass:B}, 
we obtain from \cref{thm:Kasp_prod_index,thm:red_cyl} the equalities 
\[
[\pS_{\R\times N}(\cdot)] \otimes_{C_0(\R\times N)} [\D_{\R\times N}]
= \big[ (\lambda\pS_{\R\times N})\times\D_{\R\times N} \big] 
= \big[ (\lambda\pS)\times\D \big] .
\qedhere 
\]
\end{proof}

\subsection{The `classical' Callias theorem and index pairings}
\label{sec:classical_Callias}

Let $M$, $\D$ and $\pF$ satisfy assumptions \ref{ass:A}, \ref{ass:B}, \ref{ass:C} with signature $(p,q)$, such that $(\lambda\pF)\times\D$ is Fredholm (and hence a Callias-type operator) for $\lambda \geq \lambda_0 > 0$. 
In this section we consider the special case where $A$ is a unital $C^*$-algebra and the Hilbert $A$-module $E$ is finitely generated and projective. 
In order to constantly remind the reader of this additional assumption, we denote the potential by the symbol $\pF$ instead of $\pS$. 
This is still a generalisation of the classical setting of Dirac--Schrödinger or Callias-type operators, since by introducing the auxiliary $C^*$-algebra $A$ we replace vector spaces by finitely generated projective modules over $A$, matrix-valued functions by $\mL_A(E)$-valued functions, etc. 

The main advantage of considering this special case of finitely generated projective modules $E$, is that we can reformulate the Callias Theorem in terms of index (or spectral flow) pairings which are reminiscent of various results from the literature. 

\paragraph{The even index pairing}
Let us consider the odd-odd signature $(p,q)=(1,1)$, which is the case originally considered by Callias \cite{Cal78} and which has been most extensively studied since then. 
In this signature, the class $[(\lambda\pF)\times\D] \in \KK^0(\C,A)$ corresponds to $\Index(\D-i\lambda\pF) \in \K_0(A)$. 
Under the present assumption that $A$ is unital and $E$ is finitely generated and projective, we can rephrase the Callias Theorem in terms of an even index pairing over the hypersurface $N$. We refer to \S\ref{sec:even_index_pairing} for a description of this index pairing. 

\begin{thm}
\label{thm:11_index}
Let $(\lambda\pF)\times\D$ be a Callias-type operator of signature $(p,q)=(1,1)$. 
Suppose that $A$ is unital, and that $E$ is finitely generated and projective. 
Assume furthermore that $\pF$ is sufficiently smooth, such that $P_+(\pF_N) \colon N \to \mL_A(E)$ is a differentiable function. 
Then 
\[
\Index\big( \D - i \lambda\pF \big) 
= \big[ P_+(\pF_N) \big] \otimes_{C(N)} [\D_N] 
= \Index\big( P_+(\pF_N) (\D_N)_+ P_+(\pF_N) \big) \in \K_0(A) .
\]
\end{thm}
\begin{proof}
The assumptions on $A$ and $E$ ensure that all operators on $E$ are compact. 
In particular, the operator $\pT:=-1$ is a relatively compact perturbation of each $\pS_N(y)$. 
With $P_+(\pT) = 0$ we therefore obtain
\[
\relind_0\big(P_+(\pF_N),0\big) = \Index\big( 0 \colon \Ran P_+(\pF_N) \to \{0\} \big) = \big[ P_+(\pF_N) \big] \in \K_0(A) . 
\]
We thus obtain from the Callias \cref{thm:Callias} the equality 
\[
[(\lambda\pF)\times\D] 
= \relind_0\big(P_+(\pF_N),P_+(\pT)\big) \otimes_{C(N)} [\D_N] 
= \big[ P_+(\pF_N) \big] \otimes_{C(N)} [\D_N] .
\]
By \cref{thm:even_index_pairing_unbdd}, this pairing over $N$ is given by the even index pairing. 
\end{proof}

\begin{remark}
Our assumption that $P_+(\pF_N)$ is differentiable is needed to make sure that the unbounded operator $P_+(\pF_N) (\D_N)_+ P_+(\pF_N)$ is well-behaved. 
Without this assumption, one could alternatively take the bounded transform $F_{\D_N} := \D_N (1+\D_N^2)^{-\frac12}$ and consider the index of $P_+(\pF_N) (F_{\D_N})_+ P_+(\pF_N)$ instead. 
\end{remark}

\paragraph{The odd index pairing and the spectral flow}
In the paper \cite{Bun95}, Bunke considered not only the previously studied odd-odd signature $(p,q)=(1,1)$, but also was the first to consider the even-even signature $(p,q)=(0,0)$. 
In the latter case, the $\KK$-class $[(\lambda\pF)\times\D]$ of the Callias-type operator corresponds to the index of $(\D+\lambda\pF)_+$ in $\K_0(A)$. 
Bunke showed that this index can be computed by the index of a suitable Dirac operator on $S^1\times N$, which can also be obtained from a spectral flow pairing \cite[Theorem 2.16]{Bun95}. 
Below we will provide a similar result as a consequence of our Callias Theorem, describing the index in terms of the odd index pairing (see the Appendix, \S\ref{sec:odd_index_pairing}) and the spectral flow pairing. 

\begin{thm}
\label{thm:00_sf}
Let $(\lambda\pF)\times\D$ be a Callias-type operator of signature $(p,q)=(0,0)$. 
Suppose that $A$ is unital, and that $E = E_0 \oplus E_0$ with $\Z_2$-grading $\Gamma_\pF = 1\oplus(-1)$, where $E_0$ is finitely generated and projective. 
Assume furthermore that the potential $\pF$ has the form 
\begin{align}
\label{eq:00_pF}
\pF = \mattwo{0}{U^*}{U}{0} ,
\end{align}
for a function $U \colon M \to \mL_A(E_0)$ such that $U(y)$ is unitary for all $y\in N$. 
Then 
\[
\Index\big( (\D+\lambda\pF)_+ \big) = \Index\big( P_+(\D_N) U_N P_+(\D_N) \big) = - \SF_0(\D_N,U_N^* \D_N U_N) \in \K_0(A) ,
\]
where $U_N := U|_N \colon N \to \mL_A(E_0)$ is the restriction of $U$. 
Here $\SF_0(\D_N,U_N^* \D_N U_N)$ is defined to be the spectral flow of the family $[0,1] \ni t \mapsto (1-t) \D_N + t U_N^* \D_N U_N = \D_N + t U_N^*[\D_N,U_N]$. 
\end{thm}
\begin{proof}
The assumptions on $A$ and $E$ ensure that all operators on $E$ are compact. 
We now choose the operator $\pT$ with the positive spectral projection $P_+(\pT)$ given by 
\[
\pT := \mattwo{0}{1}{1}{0} , \qquad 
P_+(\pT) = \frac12 \mattwo{1}{1}{1}{1} , 
\]
and we compute 
\[
\relind_1\big(P_+(\pF_N),P_+(\pT)\big) 
= \left[ \mattwo{1}{0}{0}{U_{P_+(\pF_N)}} \right] = \big[U_{P_+(\pF_N)}\big] . 
\]
We thus obtain from the Callias \cref{thm:Callias} the equality 
\[
[(\lambda\pF)\times\D] 
= \relind_1\big(P_+(\pF_N),P_+(\pT)\big) \otimes_{C(N)} [\D_N] 
= \big[U_{P_+(\pF_N)}\big] \otimes_{C(N)} [\D_N] .
\]
Since $\pF_N$ is self-adjoint and unitary, we have $U_{P_+(\pF_N)} = U_N$. 
We know from \cref{coro:odd_index_pairing} that the pairing between odd $\K$-theory and odd $\K$-homology can be computed by: 
\[
[U_N] \otimes_{C(N)} [\D_N] 
= \Index\big( P_+(\D_N) U_N P_+(\D_N) \big) .
\]
Finally, from \cref{coro:sf_index_PuP} we obtain 
\[
\Index\big( P_+(\D_N) U_N P_+(\D_N) \big) 
= - \SF_0(\D_N,U_N^* \D_N U_N) .
\qedhere 
\]
\end{proof}

\paragraph{The odd-even pairing}
For the two signatures with $p+q=1$, we consider classes of Dirac--Schrödinger operators in $\KK^1(\C,A) \simeq \K_1(A)$. Of course, for $A=\C$, this $\K$-theory group is trivial. However, for nontrivial $A$, also $\K_1(A)$ may be nontrivial, and we may obtain nontrivial pairings $\K_p(C(N,A)) \otimes_{C(N)} \K^q(C(N)) \to \K_1(A)$. For instance, considering the odd-even signature $(p,q)=(1,0)$ we obtain the following result. 

\begin{thm}
\label{thm:10_sf}
Let $(\lambda\pF)\times\D$ be a Callias-type operator of signature $(p,q)=(1,0)$. 
Suppose that $A$ is unital, and that $E$ is finitely generated and projective. 
Assume furthermore that $\pF$ is sufficiently smooth, such that $P_+(\pF_N) \colon N \to \mL_A(E)$ is a differentiable function. 
Then 
\[
\big[ \D + \lambda \Gamma_\D \pF \big]
= \big[ P_+(\pF_N) \D_N P_+(\pF_N) \big] \in \K_1(A) .
\]
\end{thm}
\begin{proof}
As in the proof of \cref{thm:11_index}, it follows from the Callias \cref{thm:Callias} that 
\[
[(\lambda\pF)\times\D] 
= \big[ P_+(\pF_N) \big] \otimes_{C(N)} [\D_N] .
\]
By \cref{thm:even_index_pairing_unbdd}, this pairing equals $\big[ P_+(\pF_N) \D_N P_+(\pF_N) \big]$. 
\end{proof}

Now suppose that $A = C(S^1,B)$ and $E = C(S^1,V)$ for a Hilbert $B$-module $V$. 
Then the potential $\pF$ on $C_0(M,E) \simeq C(S^1,C_0(M,V))$ can be viewed as a family of operators $\{\pF_t\}_{t\in S^1}$ on $C_0(M,V)$. 
We consider the positive spectral projection $P_+(\pF_N(t))$ of the restriction $\pF_N \equiv \pF|_N$ at the point $t\in S^1$. 

\begin{coro}
\label{coro:10_sf}
In the setting of \cref{thm:10_sf}, suppose that $A = C(S^1,B)$ for a unital $C^*$-algebra $B$, and $E = C(S^1,V)$ for a finitely generated projective module $V$ over $B$. 
Then 
\[
\SF_0 \big\{ \D + \lambda \Gamma_\D \pF_t \big\}_{t\in S^1}
= \SF_0 \big\{ P_+(\pF_N(t)) \D_N P_+(\pF_N(t)) \big\}_{t\in S^1} \in \K_0(B) .
\]
\end{coro}
\begin{proof}
The standard isomorphism $\K_1(C(S^1,B)) \xrightarrow{\simeq} \K_0(B)$ is given by the spectral flow. 
The statement then follows from \cref{thm:10_sf}. 
\end{proof}

\subsection{Cobordism invariance}
\label{sec:cobordism}

\begin{coro}
\label{coro:cobordism}
Let $(\lambda\pS)\times\D$ be a Callias-type operator with signature $(p,q)$. 
Assume that the potential $\pS$ is uniformly invertible on all of $M$. 
Then 
\[
\relind_{p+1}\big(P_+(\pS_N(\cdot)),P_+(\pT(\cdot))\big) \otimes_{C(N)} [\D_N] 
= 0 \in \K_{p+q}(A) .
\]
\end{coro}
\begin{proof}
By \cref{prop:D_S_invertible}, the invertibility of the potential implies that the class of $(\lambda\pS)\times\D$ vanishes (for $\lambda$ large enough), so the result follows from the Callias \cref{thm:Callias}. 
\end{proof}

Let us consider two special cases of the above result, using the even and odd index pairings described in \S\ref{sec:classical_Callias}. 
First, we recover the well-known cobordism invariance of the index:
\begin{coro}
Let $M$ be a connected Riemannian manifold, and let $\D$ be an essentially self-adjoint elliptic first-order differential operator on a hermitian vector bundle $\bF\to M$ (which is ungraded). 
Let $K \subset M$ be a compact subset with smooth compact boundary $N$. 
Assume that $\D$ is of product form on a collar neighbourhood of $N$ (as in \cref{defn:D_product_form}, with $q=1$), and let $\D_N$ be the corresponding $\Z_2$-graded operator on $N$. 
Then 
\[
\Index\big( (\D_N)_+ \big) = 0 \in \K_0(\C) \simeq \Z .
\]
\end{coro}
\begin{proof}
We can consider the constant potential $\pF(x) = 1$ (for all $x\in M$) acting on the Hilbert space $E=\C$, 
so that $M$, $\D$, and $\pF$ satisfy the assumptions \ref{ass:A}-\ref{ass:C} with the signature $(p,q)=(1,1)$. 
We note that $P_+(\pF_N) = 1 \in C(N)$. 
As in the proof of \cref{thm:11_index}, we can choose $\pT=-1$ so that we now have 
\[
\relind_{p+1}\big(P_+(\pF_N(\cdot)),P_+(\pT(\cdot))\big) = \big[ P_+(\pF_N) \big] = [1] .
\]
By \cref{thm:even_index_pairing_unbdd}, the pairing $[1] \otimes_{C(N)} [\D_N]$ equals $\Index\big( (\D_N)_+ \big)$, and the statement then follows from \cref{coro:cobordism}. 
\end{proof}

Of course, the above result is nothing new, and could have been proven without the elaborate methods of our generalised Callias \cref{thm:Callias}. 
Somewhat more interesting is that in our approach one can easily formulate similar cobordism invariance results also in other signatures. 
As an example, we present here the cobordism invariance of the spectral flow pairing:
\begin{coro}
\label{coro:cobordism_sf}
Let $M$ be a connected Riemannian manifold, and let $\D$ be an odd essentially self-adjoint elliptic first-order differential operator on a $\Z_2$-graded hermitian vector bundle $\bF\to M$. 
Let $K \subset M$ be a compact subset with smooth compact boundary $N$. 
Assume that $\D$ is of product form on a collar neighbourhood of $N$ (as in \cref{defn:D_product_form}, with $q=0$), and let $\D_N$ be the corresponding operator on $N$. 
Let $U_N \in C^1(N) \otimes M_n(\C)$ be unitary, such that it can be extended to a unitary differentiable function on all of $M$. 
Then 
\[
\SF_0(\D_N,U_N^* \D_N U_N) = 0 \in \K_0(\C) \simeq \Z .
\]
\end{coro}
\begin{proof}
Let $U$ be the unitary extension of $U_N$ to $M$ and set $E = \C^{\oplus n} \oplus \C^{\oplus n}$. 
Then we obtain an invertible potential $\pF = \mattwo{0}{U^*}{U}{0}$ on $C_0(M,E)$, 
so that $M$, $\D$, and $\pF$ satisfy the assumptions \ref{ass:A}-\ref{ass:C} with the signature $(p,q)=(0,0)$. 
As in the proof of \cref{thm:00_sf}, we can choose $\pT$ such that 
\[
\relind_1\big(P_+(\pF_N),P_+(\pT)\big) = \big[U_{P_+(\pF_N)}\big] = [U_N] . 
\]
The statement then follows from \cref{coro:sf_index_PuP,coro:odd_index_pairing,coro:cobordism}. 
\end{proof}

\begin{remark}
The above results in particular apply to manifolds of the form $M = K \cup_N \big( [0,\infty) \times N \big)$. 
We can view $K$ as a compact manifold with boundary $N$. 
In this case, the assumption on $U_N$ is really about being able to extend $U_N$ to a unitary on $K$ (since a unitary extension on the cylindrical end always exists). 

We also point out that the above results continue to hold for the $\K_0(A)$-valued index or spectral flow. 
\end{remark}

\section{Toeplitz operators}
\label{sec:Toeplitz}

The index of a classical Dirac--Schrödinger operator $\pF\times\D$ can also be related to the index of a corresponding Toeplitz operator (given by $\pF$ compressed to the kernel of $\D$) \cite{GH96,Bun00}. 
A similar result by Braverman \cite{Bra19}, but shifted in degree, considered the case where $\D$ is equipped with a $\Z_2$-grading and where the potential is given by a smooth matrix-valued function $\pF$ on $M\times S^1$ (or equivalently, by a family $\{\pF_t\}_{t\in S^1}$ of potentials on $M$), 
and showed that the index of the Dirac--Schrödinger operator corresponds to the \emph{spectral flow} of a \emph{family} of Toeplitz operators on $S^1$. 
Braverman applied this equality to the bulk-edge correspondence for topological insulators. 
The appendix \cite{Bra19_vdD} to Braverman's paper explains Braverman's main result in terms of the Kasparov product. 
In this section, we generalise these results and provide a unified approach for arbitrary signatures $(p,q)$. 
As in \S\ref{sec:classical_Callias}, we consider only potentials on finitely generated projective modules over unital $C^*$-algebras, and to emphasise this assumption we denote the potential by $\pF$ instead of $\pS$. 

Throughout this section, we will assume the following: 
\begin{assumption*}[(T)]
\customlabel{ass:T}{(T)}
Let $A$ be a (trivially graded) unital $C^*$-algebra, and let $E$ be a finitely generated projective Hilbert $A$-module. 
Let $M$ be a connected Riemannian manifold (typically noncompact), and let $\D$ be an essentially self-adjoint elliptic first-order differential operator on a hermitian vector bundle $\bF\to M$.
Let $\pF \colon M \to \mL_A(E)$ be a bounded differentiable function, which is uniformly invertible outside of a compact subset of $M$, such that $\pF(x)=\pF(x)^*$ for all $x\in M$. 
Moreover, we consider self-adjoint unitaries $\gradF$ and $\gradD$ as in assumption \ref{ass:A}, depending on the signature $(p,q)$. 
We assume furthermore that the following conditions are satisfied:
\begin{itemize}
\item[(T1)]
\customlabel{ass:T1}{(T1)}
Zero is an isolated point of the spectrum of $\D \equiv 1\otimes_d\D$ on $L^2(M,E\otimes\bF)$. 
\item[(T2)]
\customlabel{ass:T2}{(T2)} 
The operator $\big[\D,\pF\big]_\pm$ vanishes at infinity (in the sense that the function $M\to\R$, $x \mapsto \big\| \big[\D,\pF\big]_\pm(x) \big\|$ vanishes at infinity). 
\end{itemize}
\end{assumption*}

\begin{remark}
\label{rem:Toeplitz_Fredholm}
\begin{itemize}
\item 
We note that assumption \ref{ass:T} implies the assumptions \ref{ass:A} and \ref{ass:B} (taking $\pS \equiv \pF$), and therefore we may freely use all the results from \cref{sec:gen_DS}. In particular, using \ref{ass:T2}, we know from \cref{prop:Fredholm_all_lambda} that the product operator $\pF\times\D$ is Fredholm. 
\item 
We note that, while assumption \ref{ass:C2} is trivially satisfied, we do not consider any hypersurface $N\subset M$ and therefore do not require any product structure as in \ref{ass:C1}. 
\end{itemize}
\end{remark}

\begin{remark}
\label{rem:Toeplitz_vanishing}
In the setting of assumption \ref{ass:T}, it follows from \cref{prop:D-gap} that the Dirac--Schrödinger operators $\pF\times_{p,1}\D$ with $q=1$ represent the trivial class in $\K$-theory (this vanishing was already observed by Bunke \cite[Proposition 2.8]{Bun95} in the signature $(p,q)=(1,1)$). 
For this reason, we will ignore the case $q=1$ in the following definition. 
For the two signatures $(p,q)=(0,0)$ and $(p,q)=(1,0)$ we will see nontrivial examples in \S\ref{sec:Toeplitz_00} and \S\ref{sec:Toeplitz_10}. 
\end{remark}

\begin{defn}
\label{defn:Toeplitz}
Let $\mN \subset L^2(M,E\otimes\bF)$ denote the kernel of $\D \equiv 1\otimes_d\D$, and consider the corresponding orthogonal projection denoted by $P_0 \equiv P_0(\D) \colon L^2(M,E\otimes\bF) \to \mN$. 
We define the \emph{Toeplitz operator} $T_\pF \colon \mN \to \mN$ by compressing the product operator onto the range of $P_0$:
\[
T_\pF := P_0 \circ (\pF\times\D) \colon \mN \to L^2(M,E\otimes\bF) \to \mN . 
\]
We have the following explicit formulas:
\[
\begin{aligned}
\text{even-even signature } (p,q)=(0,0): \quad& T_\pF = P_0 \pF P_0 ;\\
\text{odd-even signature } (p,q)=(1,0): \quad& T_\pF = P_0 \gradD\pF P_0 . 
\end{aligned}
\]
For convenience, we shall sometimes write $T_\pF = P_0 \til\pF P_0$, with 
$\til\pF := \pF$ if $(p,q)=(0,0)$ and $\til\pF := \Gamma_\D \pF$ if $(p,q)=(1,0)$. 
\end{defn}

\begin{remark}
\label{rem:Toeplitz_trivial}
Our main aim below is to prove that $[\pF \times_{p,q} \D] = [T_\pF] \in \K_{p+q}(A)$ for the case $q=0$. 
The same equality in fact trivially also holds for the case $q=1$. 
Indeed, as mentioned in \cref{rem:Toeplitz_vanishing}, we have for $q=1$ the vanishing result $[\pF \times_{p,1} \D] = 0 \in \K_{p+1}(A)$, and we can show directly that the class of the corresponding Toeplitz operator also vanishes (to be precise, one should first check that the Toeplitz operator is self-adjoint and Fredholm, but this follows as in the proof of \cref{thm:Toeplitz}). 

First, in the odd-odd signature $(1,1)$, the Toeplitz operator is given on $\mN\oplus\mN$ by 
\[
T_\pF = \mattwo{0}{i P_0 \pF P_0}{-i P_0 \pF P_0}{0} .
\]
Its class in $\K$-theory is given by $\Index(-i P_0 \pF P_0) = \Index(P_0\pF P_0) \in \K_0(A)$. Since $P_0 \pF P_0$ is self-adjoint, this index vanishes, so we have $[T_\pF] = 0$. (Alternatively, we can observe that $T_\pF$ has a Clifford symmetry which interchanges the two copies of $\mN$, and obtain the vanishing of its class from \cref{lem:Fred_Cliff-symm}.)

Second, in the even-odd signature $(0,1)$, the Toeplitz operator is given by 
\[
T_\pF = P_0 \pF P_0 
= \mattwo{0}{P_0 \pF_- P_0}{P_0 \pF_+ P_0}{0} .
\]
We now observe that $\Gamma_\pF$ yields a Clifford symmetry for $T_\pF$, so by \cref{lem:Fred_Cliff-symm} we again have $[T_\pF] = 0$. 
\end{remark}

\begin{prop}
\label{prop:cpt_comm_P_pS}
The commutator $[P_0,\pF]$ on $L^2(M,E\otimes\bF)$ is compact. 
\end{prop}
\begin{proof}
The proof is similar to the proofs of \cite[Lemma 2.4]{Bun00} and \cite[Lemma 5.2]{Bra19}, but we take care that it works in both signatures. 

First, by Rellich's Lemma, the operator $\D$ on $L^2(M,\bF)$ has locally compact resolvents. 
Since $A$ is unital and $E$ is finitely generated projective, the identity operator on $E$ is compact, and it follows that $\D \equiv 1\otimes_d\D$ on $L^2(M,E\otimes\bF)$ also has locally compact resolvents. 
We compute 
\begin{align*}
\big[ \pF , (1+\D^2)^{-1} \big] 
&= (1+\D^2)^{-1} \big( \D [\D,\pF]_\pm - (-1)^p [\D,\pF]_\pm \D \big) (1+\D^2)^{-1} .
\end{align*}
Since $\D$ has locally compact resolvents and $[\D,\pF]_\pm$ vanishes at infinity (assumption \ref{ass:T2}), it follows that $[\D,\pF]_\pm (1+\D^2)^{-1}$ and therefore also $\big[ \pF , (1+\D^2)^{-1} \big]$ is compact. 

By assumption \ref{ass:T1}, there exists $\epsilon>0$ such that $P_0 = \phi(\D)$ for any even function $\phi\in C_c(-\epsilon,\epsilon)$ with $\phi(0)=1$. 
By Stone--Weierstrass, the subalgebra of even functions in $C_0(\R)$ is generated by the function $t \mapsto (1+t^2)^{-1}$. 
Thus it follows that also $[P_0,\pF] = [\phi(\D),\pF]$ is compact. 
\end{proof}

\begin{thm}
\label{thm:Toeplitz}
Consider the setting of assumption \ref{ass:T} with signature $(p,q)$. 
Then the Toeplitz operator $T_\pF$ is a (bounded) self-adjoint Fredholm operator, 
and we have the equality
\[
[T_\pF] = [\pF\times_{p,q}\D] \in \KK^{p+q}(\C,A) . 
\]
\end{thm}
\begin{proof}
Since the case $q=1$ is trivial (see \cref{rem:Toeplitz_trivial}), we focus on the case $q=0$ and use the notation of \cref{defn:Toeplitz}. 
Nevertheless, the proof also works in the signatures $(p,1)$, requiring at most some notational changes. 

Since $\til\pF$ is bounded and self-adjoint in each signature, the Toeplitz operator $T_\pF = P_0 \til\pF P_0$ is also bounded and self-adjoint. 
We mentioned in \cref{rem:Toeplitz_Fredholm} that $\pF\times\D$ is Fredholm. 
We need to show that $\pF\times\D$ and $T_\pF = P_0 (\pF\times\D) P_0$ represent the same class in $\KK^{p+q}(\C,A)$. 
Let us write $Q := 1-P_0$. 
From \cref{prop:cpt_comm_P_pS}, we know that 
\[
P_0 (\pF\times\D) Q = P_0 \til\pF Q = \big[ P_0 , \til\pF \big] Q = \begin{cases} \big[ P_0 , \pF \big] Q , & (p,q)=(0,0) , \\ \Gamma_\D \big[ P_0 , \pF \big] Q , & (p,q)=(1,0) \end{cases}
\]
is compact. Similarly, also $Q (\pF\times\D) P_0$ is compact. 
In particular, $P_0 (\pF\times\D) P_0 \oplus Q (\pF\times\D) Q$ is a compact perturbation of the Fredholm operator $\pF\times\D$, which implies that both $T_\pF = P_0 (\pF\times\D) P_0$ and $Q (\pF\times\D) Q$ are Fredholm, 
and that 
\[
[\pF\times\D] = [P_0 (\pF\times\D) P_0] + [Q (\pF\times\D) Q] = [T_\pF] + [Q (\pF\times\D) Q] . 
\]
Rescaling $\pF$ by a positive constant $\lambda>0$ and using \cref{prop:Fredholm_all_lambda}, we find similarly that $Q \big( (\lambda\pF)\times\D \big) Q$ is Fredholm for any $\lambda>0$. Moreover, in the case $\lambda=0$ we know that $Q \big( 0\times\D \big) Q$ is invertible (on the complement of the kernel of $\D$). Thus $Q (\pF\times\D) Q$ is homotopic (within the space of Fredholm operators with domain $\Dom\D$) to an invertible operator, and therefore $[Q (\pF\times\D) Q] = 0$. 
\end{proof}

\subsection{The even-even signature}
\label{sec:Toeplitz_00}

Toeplitz operators with the even-even signature $(p,q)=(0,0)$ were considered by Bunke in \cite{Bun00}. 
(In fact, Bunke also considered the odd-odd signature $(p,q)=(1,1)$, which he showed to be trivial, as mentioned in \cref{rem:Toeplitz_vanishing}.) 
Here we shall prove a generalisation of \cite[Proposition 2.6]{Bun00}. 

\begin{prop}
\label{prop:Toeplitz_00}
Let assumption \ref{ass:T} with signature $(p,q)=(0,0)$ be satisfied by the operators 
\[
\D = \mattwo{0}{\D_-}{\D_+}{0}
\quad\text{and}\quad 
\pF = \mattwo{0}{f^*}{f}{0} .
\]
Define $T_f^\pm := P_0^\pm f P_0^\pm$, where $P_0^\pm \equiv P_0(\D_\pm)$ denotes the projection onto the kernel of $\D_\pm$. 
Then the image of $[\pF\times_{0,0}\D] \in \KK^0(\C,A)$ in $\K_0(A)$ under the standard isomorphism is given by 
\[
[\pF\times_{0,0}\D] 
\xmapsto{\simeq} \Index \left( \mattwo{f}{-\D_-}{\D_+}{f^*} \right) 
= \Index(T_f^+) - \Index(T_f^-) 
\in \K_0(A) .
\]
\end{prop}
\begin{proof}
On the decomposition 
\[
L^2(M,E\otimes\bF) = L^2(M,E_+\otimes\bF_+) \oplus L^2(M,E_-\otimes\bF_-) \oplus L^2(M,E_-\otimes\bF_+) \oplus L^2(M,E_+\otimes\bF_-) ,
\]
we can write 
\begin{align*}
\pF\times_{0,0}\D &= \D + \pF = \matfour{0&0&f^*&\D_-}{0&0&-\D_+&f}{f&-\D_-&0&0}{\D_+&f^*&0&0} , & 
P_0 &= \matfour{P_0^+&0&0&0}{0&P_0^-&0&0}{0&0&P_0^+&0}{0&0&0&P_0^-} . 
\end{align*}
Under the standard isomorphism $\KK^0(\C,A) \xrightarrow{\simeq} \K_0(A)$, the class $[\pF\times_{0,0}\D] \in \KK^0(\C,A)$ is given by the index of $(\pF\times_{0,0}\D)_+ = \mattwo{f}{-\D_-}{\D_+}{f^*}$. 
Furthermore, from \cref{thm:Toeplitz} we have the equality $[T_\pF] = [(\pF\times_{0,0}\D)]$, and therefore the index of $(\pF\times_{0,0}\D)_+$ is equal to the index of 
\[
\mattwo{P_0^+}{0}{0}{P_0^-} \mattwo{f}{-\D_-}{\D_+}{f^*} \mattwo{P_0^+}{0}{0}{P_0^-} 
= \mattwo{P_0^+ f P_0^+}{0}{0}{P_0^- f^* P_0^-} . 
\]
The index of the latter is indeed given by 
\begin{align*}
\Index( P_0^+ f P_0^+ ) + \Index( P_0^- f^* P_0^- ) 
&= \Index( P_0^+ f P_0^+ ) - \Index( P_0^- f P_0^- ) \\
&= \Index(T_f^+) - \Index(T_f^-) .
\qedhere
\end{align*}
\end{proof}

\begin{remark}
In the special case where $P_0^-=0$, we obtain the equality 
\[
\Index\left( \mattwo{f}{-\D_-}{\D_+}{f^*} \right) = \Index(P_0^+fP_0^+) . 
\]
This situation occurs for instance in the setting of \cite{GH96}, where $M$ is a strongly pseudoconvex domain in $\C^n$ with the Dolbeault--Dirac operator $\D$ (see \S\ref{sec:pseudoconvex} below). 
\end{remark}

\subsection{The odd-even signature}
\label{sec:Toeplitz_10}

Now let us consider the odd-even signature $(p,q)=(1,0)$. 
Since $p+q=1$, we therefore consider classes of Dirac--Schrödinger operators in $\KK^1(\C,A) \simeq \K_1(A)$. Of course, for $A=\C$, this $\K$-theory group is trivial. In order to obtain a nontrivial index, the easiest example to consider would be to take $A = C(S^1)$. Indeed, such a setting has already been considered by Braverman \cite{Bra19}, who related the index of a Dirac--Schrödinger operator to the spectral flow of a family of Toeplitz operators (parametrised by the circle). 

Let us introduce some notation. 
Since $q=0$, the operator $\D$ on the bundle $\bF = \bF_+ \oplus \bF_- \to M$ is $\Z_2$-graded, and we write 
\[
\D = \mattwo{0}{\D_-}{\D_+}{0} .
\]
We denote again by $P_0^\pm \equiv P_0(\D_\pm)$ the projection onto the kernel of $\D_\pm$. 

Now suppose that $A = C(S^1,B)$ and $E = C(S^1,V)$ for a Hilbert $B$-module $V$. 
Then the potential $\pF$ on $C_0(M,E) \simeq C(S^1,C_0(M,V))$ can be viewed as a family of operators $\{\pF_t\}_{t\in S^1}$ on $C_0(M,V)$. 
We introduce the corresponding Toeplitz operators 
\[
T_t^\pm := P_0^\pm \pF_t P_0^\pm 
\quad\text{on}\quad 
\Ker\D_\pm \subset L^2(M,V\otimes\bF) .
\]
The following result is a generalisation of \cite[Theorem 2.1]{Bra19} (see also the reformulation in \cite[Proposition A.1]{Bra19}). 

\begin{prop}
\label{prop:Toeplitz_10}
Consider the setting of assumption \ref{ass:T} with signature $(p,q) = (1,0)$. 
Assume that $A = C(S^1,B)$ for some (trivially graded) unital $C^*$-algebra $B$, and that $E = C(S^1,V)$ for some finitely generated projective module $V$ over $B$.
Then the image of $[\pF \times_{1,0} \D] \in \KK^1(\C,C(S^1,B))$ in $\K_0(B)$ under the standard isomorphism is given by 
\[
[\pF \times_{1,0} \D] 
\xmapsto{\simeq} \SF_0\big(\{\pF_t \times_{1,0} \D\}_{t\in S^1}\big)
= \SF_0\big(\{T_t^+\}_{t\in S^1}\big) - \SF_0\big(\{T_t^-\}_{t\in S^1}\big) . 
\]
\end{prop}
\begin{proof}
We recall that the product operator is an ungraded operator of the form
\[
\pF\times_{1,0}\D := \D + \gradD \pF = \mattwo{\pF}{\D^-}{\D^+}{-\pF} , 
\]
describing a class $[\pF\times_{1,0}\D] \in \KK^1(\C,A)$. 
From \cref{thm:Toeplitz} we have $[\pF\times_{1,0}\D] = [T_\pF]$, where the Toeplitz operator $T_\pF$ is of the form 
\[
T_\pF = P_0 (\gradD \pF) P_0 = \mattwo{P_0^+ \pF P_0^+}{0}{0}{- P_0^- \pF P_0^-} .
\]
This operator $T_\pF$ on $\mN = E\otimes\Ker\D \simeq C(S^1,V\otimes\Ker\D)$ is given by a family of operators $\{T_{\pF_t}\}_{t\in S^1}$ on the Hilbert $B$-module $V\otimes\Ker\D = (V\otimes\Ker\D_+)\oplus(V\otimes\Ker\D_-)$, where 
\[
T_{\pF_t} = \mattwo{P_0^+ \pF_t P_0^+}{0}{0}{- P_0^- \pF_t P_0^-} = \mattwo{T_t^+}{0}{0}{- T_t^-} .
\]
Under the isomorphism $\KK^1(\C,C(S^1,B)) \xrightarrow{\simeq} \K_0(B) $ given by the spectral flow, the $\KK$-classes of $\pF\times_{1,0}\D$ and $T_\pF$ correspond to 
\begin{align*}
[\pF\times_{1,0}\D] 
&\xmapsto{\simeq} \SF_0 \big( \{\pF_t\times_{1,0}\D\}_{t\in S^1} \big) , \\ 
[T_\pF] 
&\xmapsto{\simeq} \SF_0 \left( \left\{ \mattwo{T_t^+}{0}{0}{- T_t^-} \right\}_{t\in S^1} \right) 
= \SF_0 \big( \{T_t^+\}_{t\in S^1} \big) - \SF_0 \big( \{T_t^-\}_{t\in S^1} \big) .
\qedhere
\end{align*}
\end{proof}

\begin{remark}
In the special case where $P_0^-=0$, the above proposition yields the equality $[\pF \times_{1,0} \D] \xmapsto{\simeq} \SF_0\big(\{P_0^+ \pF_t P_0^+\}_{t\in S^1}\big)$. 
This situation occurs for instance for the Dolbeault--Dirac operator on a strongly pseudoconvex domain, as described in \cite[\S3]{Bra19} (see also \S\ref{sec:pseudoconvex} below). 
\end{remark}

\subsection{Toeplitz theorems}
\label{sec:Toeplitz_Callias}

The next two results combine our \cref{thm:Toeplitz} on Toeplitz operators with our Callias \cref{thm:Callias}. 
Thus, let $M$, $\D$ and $\pF$ now satisfy assumptions \ref{ass:T} and \ref{ass:C} with signature $(p,0)$, and let $N\subset M$ denote the hypersurface from assumption \ref{ass:C}. 
As explained in \cref{rem:Toeplitz_Fredholm}, we know from \cref{prop:Fredholm_all_lambda} that $\pF\times\D$ is a Callias-type operator. 

Let us first consider the even-even signature $(p,q) = (0,0)$. 
We then obtain the following index equality between the Toeplitz operators constructed from $\pF$ and $P_0(\D)$ on the one hand, and $\pF_N \equiv \pF|_N$ and $P_+(\D_N)$ on the other hand. 
\begin{coro}[Index theorem for Toeplitz operators]
\label{coro:Toeplitz_index}
Let assumptions \ref{ass:T} and \ref{ass:C} with signature $(0,0)$ be satisfied by the operators 
\[
\D = \mattwo{0}{\D_-}{\D_+}{0}
\quad\text{and}\quad 
\pF = \mattwo{0}{U^*}{U}{0} .
\]
Assume furthermore that $U_N \equiv U|_N$ is unitary for all $y\in N$. 
Then, under the standard isomorphism, the image of $[\pF\times_{0,0}\D] \in \KK^0(\C,A)$ in $\K_0(A)$ is given by 
\[
\Index\big( P_0(\D_+) U P_0(\D_+) \big) - \Index\big( P_0(\D_-) U P_0(\D_-) \big)
= \Index\big( P_+(\D_N) U_N P_+(\D_N) \big) .
\]
\end{coro}
\begin{proof}
The statement follows immediately from \cref{thm:00_sf,prop:Toeplitz_00}. 
\end{proof}

Next, let us consider the odd-even signature $(p,q) = (1,0)$. 
We now obtain a spectral flow equality for families of Toeplitz operators on $M$ and $N$, as follows. 
\begin{coro}[Spectral flow theorem for Toeplitz operators]
\label{coro:Toeplitz_sf}
Assume that $A = C(S^1,B)$ for some (trivially graded) unital $C^*$-algebra $B$, and that $E = C(S^1,V)$ for some finitely generated projective module $V$ over $B$.
Let assumptions \ref{ass:T} and \ref{ass:C} with signature $(p,q) = (1,0)$ be satisfied by the operators 
\[
\D = \mattwo{0}{\D_-}{\D_+}{0}
\quad\text{and}\quad 
\pF = \{\pF_t\}_{t\in S^1} .
\]
Assume furthermore that $\pF$ is sufficiently smooth, so that $P_+(\pF_N(t)) \colon N \to \mL_A(E)$ is differentiable for each $t\in S^1$, where $P_+(\pF_N(t))$ is the positive spectral projection of the restriction $\pF_N \equiv \pF|_N$ at the point $t\in S^1$. 
Then, under the standard isomorphism, the image of $[\pF \times_{1,0} \D] \in \KK^1(\C,C(S^1,B))$ in $\K_0(B)$ is given by 
\begin{multline*}
\SF_0\big( \big\{P_0(\D_+) \pF_t P_0(\D_+) \big\}_{t\in S^1} \big) - \SF_0\big( \big\{P_0(\D_-) \pF_t P_0(\D_-) \big\}_{t\in S^1} \big) \\
= \SF_0\big( \big\{ P_+(\pF_N(t)) \D_N P_+(\pF_N(t)) \big\}_{t\in S^1} \big) 
\in \K_0(B) .
\end{multline*}
\end{coro}
\begin{proof}
Noting that the standard isomorphism $\K_1(C(S^1,B)) \simeq \K_0(B)$ is given by the (even) spectral flow, the statement follows from \cref{thm:10_sf,prop:Toeplitz_10}. 
\end{proof}

\begin{remark}
\cref{coro:Toeplitz_sf} yields a spectral flow analogue of the index equality from \cref{coro:Toeplitz_index}, with the notable difference that on the right-hand side, the potential $\pF_N$ and the Dirac operator $\D_N$ have swapped roles (the positive spectral projection now comes from the potential $\pF_N$ instead of from the elliptic operator $\D_N$). 
\end{remark}

\subsection{Strongly pseudoconvex domains}
\label{sec:pseudoconvex}

For the Dolbeault--Dirac operator on a strongly pseudoconvex domain, the kernel of $\D_-$ is trivial. In this case, \cref{coro:Toeplitz_index} yields the equality 
\[
\Index\big( P_0(\D_+) U P_0(\D_+) \big) 
= \Index\big( P_+(\D_N) U_N P_+(\D_N) \big) .
\]
Such an equality is indeed known (see e.g.\ \cite[Theorem 11.7.5]{Higson-Roe00}) and holds more generally for Dirac operators on compact Spin$^c$ manifolds with boundary (\cite[Proposition 6.4]{BE21}). 
To provide a concrete example of our results, we consider here the setting of Toeplitz operators on a strongly pseudoconvex domain. These have been studied previously in \cite{GH96} (for the signature $(0,0)$) and \cite{Bra19} (for the signature $(1,0)$), to which we refer for more details and further references (see also the exposition in \cite[\S11.7]{Higson-Roe00}). 

Consider a strongly pseudoconvex domain $\overline{M} \subset \C^n$ with smooth boundary $N := \partial M$. 
Its interior $M$ is a complete Riemannian manifold equipped with the Bergman metric. 
We consider the Dolbeault--Dirac operator $\D = \bar\partial + \bar\partial^*$ on the space $\Omega^{(n,*)}(M)$ of differential forms of type $(n,*)$, which is $\Z_2$-graded: 
\[
\Omega^{(n,*)}(M) = \bigoplus_{k=0}^n \Omega^{(n,k)}(M) = \Big( \bigoplus_{k \textrm{ even}} \Omega^{(n,k)}(M) \Big) \oplus \Big( \bigoplus_{k \textrm{ odd}} \Omega^{(n,k)}(M) \Big) .
\]
The Dolbeault--Dirac operator satisfies $\ker(\D) = \ker(\D_+) \subset \Omega^{(n,0)}(M)$, and the projection $P_0(\D) = P_0(\D_+)$ onto the kernel of $\D$ is precisely the projection onto the space $\Ker\left( \bar\partial|_{\Omega^{(n,0)}} \right)$ of holomorphic forms in (the Hilbert space completion of) $\Omega^{(n,0)}(M)$. 
In the following two examples, we consider Toeplitz operators on the strongly pseudoconvex domain $\overline{M}$ in both of the nontrivial signatures. 

\paragraph{Signature $(p,q)=(0,0)$.}
Let $A=\C$ and $E=\C^{2l}$, and consider a matrix-valued potential $\pF \colon \overline{M} \to M_{2l}(\C)$ satisfying assumption \ref{ass:T} with the even-even signature $(p,q)=(0,0)$.  
In particular, $\pF$ is $\Z_2$-graded and can be written as 
\[
\pF = \mattwo{0}{f^*}{f}{0} 
\]
for some $M_l(\C)$-valued differentiable function $f$ on $\overline{M}$ which is invertible near the boundary $N$. 
Assuming for simplicity that the restriction $f_N := f|_N$ is unitary, we obtain from \cref{coro:Toeplitz_index} the equality 
\begin{align}
\label{eq:Toeplitz_index_Dolbeault}
\Index(P_0(\D) f P_0(\D)) 
= \Index\big( P_+(\D_N) f_N P_+(\D_N) \big) 
\in \Z ,
\end{align}
using that the Dolbeault--Dirac operator satisfies $P_0(\D) = P_0(\D_+)$ and $P_0(\D_-) = 0$. 
Thus we recover the equality of the index of the Toeplitz operator on the interior $M$ with the index of the Toeplitz operator on the boundary $\partial M$, as stated in \cite[Theorem 11.7.5]{Higson-Roe00}. 

\begin{example}
Consider now the special case where $M=\bD$ is the open unit disk in $\C$ with boundary $N=S^1$, and let $l=1$. 
The boundary operator $\D_N$ corresponds to the standard Dirac operator $-i \partial_\theta$ on the circle $N=S^1$. Thus, the index of the Toeplitz operator $T_f^+ = P_0(\D) f P_0(\D)$ on $\bD$ (given by compressing multiplication by $f$ to the holomorphic $L^2$-functions on $\bD$) equals the index of the classical Toeplitz operator $P_+ f|_{S^1} P_+$ appearing in the well-known Gohberg--Krein index formula (where $P_+ \equiv P_+(-i\partial_\theta)$ denotes the Szegö projection onto the Hardy space $\mH^2(S^1) \subset L^2(S^1)$). 

For instance, if $f(z) = z$, then the kernel of $T_z^+$ is trivial (if $\varphi$ is a holomorphic function, then so is $z\cdot\varphi$), whereas the cokernel is one-dimensional (it consists of the constant functions). Hence $\Index T_z^+ = -1$. 
Similarly, on the space $\Ran P_+ = \overline\Span\{z^k \mid k\geq0\}$, multiplication by $z$ is the unilateral right shift. Again, the kernel of $P_+ z|_{S^1} P_+$ is trivial and the cokernel consists of the constant functions, so also $\Index P_+ f|_{S^1} P_+ = -1$. 
\end{example}

\paragraph{Signature $(p,q)=(1,0)$.}
Now let $A=C(S^1)$ and $E = C(S^1,\C^l)$, and consider a smooth family $\pF = \{\pF_t\}_{t\in S^1}$ of matrix-valued potentials $\pF_t\colon \overline{M} \to M_l(\C)$ satisfying assumption \ref{ass:T} with the odd-even signature $(p,q)=(1,0)$. 
Then from \cref{coro:Toeplitz_sf} we obtain 
\begin{align}
\label{eq:Toeplitz_sf_Dolbeault}
\SF_0\big( \big\{ P_0(\D) \pF_t P_0(\D) \big\}_{t\in S^1} \big) 
= \SF_0\big( \big\{ P_+(\pF_N(t)) \D_N P_+(\pF_N(t)) \big\}_{t\in S^1} \big) 
\in \Z ,
\end{align}
using that the Dolbeault--Dirac operator satisfies $P_0(\D) = P_0(\D_+)$ and $P_0(\D_-) = 0$. 
This yields a spectral flow analogue of the Toeplitz index equality \eqref{eq:Toeplitz_index_Dolbeault}, with the notable difference that on the right-hand side, the potential $\pF_N$ and the Dirac operator $\D_N$ have swapped roles. 

\begin{example}
Consider again the special case where $M=\bD$ is the open unit disk in $\C$ with boundary $N=S^1$, and now let $l=2$. 
Consider smooth functions $\chi_\pm \colon \R \to [0,1]$ with the following properties:
\begin{itemize}
\item $\supp \chi_+ \subset [\frac12,\infty)$ and $\supp \chi_- \subset (-\infty,\frac12]$, and 
\item $\chi_+(t) = 1$ for $t\geq1$ and $\chi_-(t) = 1$ for $t\leq0$.
\end{itemize}
We define a third function $\chi_0 \in C_c^\infty(\R)$ by $\chi_0 := \sqrt{1-\chi_+^2-\chi_-^2}$.
Now let the (ungraded) potential $\pF \in M_2\big(C^\infty([0,1],C^\infty(M))\big)$ be given by 
\[
\pF = \mattwo{\chi_-^2 + \chi_+^2 - \chi_0^2}{2\chi_0 (\chi_- + \chi_+ z)}{2\chi_0 (\chi_- + \chi_+ \bar z)}{\chi_0^2 - \chi_-^2 - \chi_+^2} . 
\]
Since $\pF(0) = \pF(1)$, we may also view $\pF$ as an element in $M_2\big(C^\infty(S^1,C^\infty(M))\big)$. 
The restriction $\pF_N$ to the boundary $N = S^1$ is a self-adjoint unitary, and we have 
\[
P_+(\pF_N) = \mattwo{\chi_-^2+\chi_+^2}{\chi_0(\chi_-+\chi_+ z)}{\chi_0(\chi_-+\chi_+ \bar z)}{\chi_0^2} . 
\]
(It is no coincidence that this projection resembles the projection $P_u$ from \cref{lem:proj_P_u}.) 
By writing $z = e^{i\theta}$, we observe that $\bar z [\D_N,z] = 1$. 
By a similar calculation as in \cref{eq:P_u-D-P_u}, we then find that the family $\big\{ P_+(\pF_N(t)) \D_N P_+(\pF_N(t)) \big\}_{t\in[0,1]}$ is unitarily equivalent to the family $\big\{ \D_N + \chi_+^2 \big\}_{t\in[0,1]}$, and therefore 
\[
\SF_0\big( \big\{ P_+(\pF_N(t)) \D_N P_+(\pF_N(t)) \big\}_{t\in[0,1]} \big) 
= \SF_0\big( \big\{ \D_N + \chi_+^2 \big\}_{t\in[0,1]} \big) 
= 1 . 
\]

Next, let us compute the spectral flow of 
\[
\hat\pF := P_0(\D) \pF P_0(\D) = \mattwo{\chi_-^2 + \chi_+^2 - \chi_0^2}{2\chi_0 \big(\chi_- + \chi_+ P_0(\D) z P_0(\D)\big)}{2\chi_0 \big(\chi_- + \chi_+ P_0(\D) \bar z P_0(\D)\big)}{\chi_0^2 - \chi_-^2 - \chi_+^2} .
\]
The range of $P_0(\D)$ consists of the holomorphic $L^2$-functions on $\bD$, so we may write $\mH := \Ran P_0(\D) = \overline\Span\{z^k \mid k\geq0\}$. 
On the orthonormal basis $\{e_k\}_{k=0}^\infty$ of $\mH$ given by $e_k(z) := \sqrt{\frac{k+1}{\pi}} z^k$, we can consider the unilateral right shift $S \colon e_k \mapsto e_{k+1}$ and the diagonal operator $K \colon e_k \mapsto k e_k$. Then we can write 
\[
P_0(\D) z P_0(\D) = \sqrt{\frac{K}{K+1}} S 
\quad\text{and}\quad 
P_0(\D) \bar z P_0(\D) = \sqrt{\frac{K+1}{K+2}} S^* . 
\]
For $t\in[0,\frac12]$ we have $\chi_+(t)=0$, and we observe that $\hat\pF(t)$ has a Clifford symmetry
\[
\gamma = \mattwo{0}{i}{-i}{0} .
\]
Hence, by \cref{lem:Fred_Cliff-symm} we have $\SF_0 \big( \big\{ \hat\pF(t) \big\}_{t\in[0,\frac12]} \big) = 0$. 
For $t\in[\frac12,1]$ we have $\chi_-(t)=0$, and since $S^* e_0 = 0$ we see that 
\[
v = \begin{pmatrix}e_0\\0\end{pmatrix} 
\]
is an eigenvector of $\hat \pF(t)$ with eigenvalue $\lambda(t) = \chi_+(t)^2 - \chi_0(t)^2$, where $\lambda(\frac12) = -1$ and $\lambda(1) = 1$. This eigenvector therefore contributes $+1$ to the spectral flow of $\hat\pF$. 
On the orthogonal complement $v^\perp = (e_0)^\perp \oplus \mH$, the operator $\hat\pF(t)|_{v^\perp}$ is invertible, so that there is no further spectral flow. 
Hence we have computed 
\[
\SF_0\big( \big\{ P_0(\D) \pF_t P_0(\D) \big\}_{t\in[0,1]} \big) 
= \SF_0\big( \big\{ \hat\pF(t) \big\}_{t\in[\frac12,1]} \big) 
= 1 , 
\]
and we have explicitly verified the spectral flow equality \eqref{eq:Toeplitz_sf_Dolbeault} in this concrete example. 
\end{example}

\appendix

\section{Fredholm operators and spectral flow}
\label{sec:prelim}

\subsection{Fredholm operators and \KK-theory}
\label{sec:Fredholm_KK} 

An adjointable operator $F \in \mL_A(E)$ is called \emph{Fredholm} if there exists a \emph{parametrix} $G \in \mL_A(E)$ such that $GF - 1$ and $FG - 1$ are compact operators on $E$. 
If $F$ is Fredholm, we denote by $\Index(F) \in \K_0(A)$ the $\K_0(A)$-valued index of $F$; for the definition of this index, we refer to \cite[\S2.2]{vdD19_Index_DS} and references therein. 

A regular operator $\D$ on a Hilbert $A$-module $E$ is called \emph{Fredholm} if there exist a \emph{left parametrix} $Q_l$ and a \emph{right parametrix} $Q_r$ such that (the closure of) $Q_l\D - 1$ and $\D Q_r - 1$ are compact endomorphisms on $E$. Then also the bounded transform $\D(1+\D^*\D)^{-\frac12}$ is Fredholm (see \cite[Lemma 2.2]{Joa03}), and we define $\Index(\D) := \Index(\D(1+\D^*\D)^{-\frac12})$. 

An odd resp.\ even regular self-adjoint Fredholm operator $\D$ on $E$ yields a well-defined class $[\D]$ in $\KK^0(\C,A)$ resp.\ $\KK^1(\C,A)$; we refer to \cite[\S2.2]{vdD19_Index_DS} for the construction of this class. 

\begin{prop}[{\cite[Proposition 2.14]{vdD19_Index_DS}}]
\label{prop:Fred_KK}
An odd resp.\ even, regular self-adjoint Fredholm operator $\D$ on a $\Z_2$-graded Hilbert $A$-module $E$ yields a well-defined class $[\D]$ in $\KK^0(\C,A)$ resp.\ $\KK^1(\C,A)$. 
Furthermore, if two such operators $\D$ and $\D'$ are homotopic, then $[\D]=[\D']$. 
\end{prop}

We consider the following standard isomorphisms (see \cite[\S2.1]{Wah07}) of $\KK^\bullet(\C,A)$ with $\K_\bullet(A)$. 
In the even ($\Z_2$-graded) case, consider an odd regular self-adjoint Fredholm operator 
\[
\D = \mattwo{0}{\D_-}{\D_+}{0} 
\]
on the standard $\Z_2$-graded Hilbert module $\hat\mH_A = \mH_A\oplus\mH_A$. The standard isomorphism $\KK^0(\C,A) \to \K_0(A)$ maps the class $[\D]$ to the Fredholm index of $\D_+$. 
In the odd case, the standard isomorphism $\KK^1(\C,A) \to \K_1(\mK_A(\mH_A)) \simeq \K_1(A)$ assigns to the class represented by a bounded Kasparov module $(\mH_A,T)$ the element $\left[ e^{i\pi(T+1)} \right] \in \K_1(\mK_A(\mH_A))$. 

Let $T$ be a regular self-adjoint operator on $E$. 
A densely defined operator $R$ on $E$ is called \emph{relatively $T$-compact} if $\Dom(T) \subset \Dom(R)$ and $R(T\pm i)^{-1}$ is compact. 
For more details on relatively compact operators in the setting of Hilbert $C^*$-modules, we refer the reader to \cite[Appendix \S A.3]{vdD25_Callias}. 

\begin{prop}[{\cite[Proposition A.11]{vdD25_Callias}}]
\label{prop:KK-class_rel_cpt_pert}
Let $T$ be a regular self-adjoint Fredholm operator on $E$, and let $R$ be a symmetric operator on $E$ which is relatively $T$-compact. 
Then: 
\begin{enumerate}
\item 
\label{item:rel_cpt_pert_reg_sa}
$T+R$ is also regular, selfadjoint, and Fredholm, and any parametrix for $T$ is also a parametrix for $T+R$. 
\item 
$[T+R] = [T] \in \KK^p(\C,A)$ (where $p=0$ if $R,T$ are odd, and $p=1$ otherwise). 
\end{enumerate}
\end{prop}

\subsubsection{Clifford symmetries}

We show next that the presence of a `Clifford symmetry' implies that a Fredholm operator is homotopically trivial. 
The following definition is adapted from \cite[Definition 4.12]{DM20}. 

\begin{defn}
Let $\D$ be a regular self-adjoint Fredholm operator on a Hilbert $A$-module $E$. 
Then $\D$ is called \emph{Clifford symmetric} if there exists a self-adjoint unitary $\gamma\in\mL_A(E)$ such that $\gamma \cdot \Dom\D \subset \Dom\D$ and $\gamma\D=-\D\gamma$. 
If $E$ is $\Z_2$-graded and $\D$ is odd, then we require in addition that $\gamma$ is also odd. 
We refer to $\gamma$ as the \emph{Clifford symmetry} of $\D$. 
\end{defn}

\begin{lem}
\label{lem:Fred_Cliff-symm}
If $\D$ is a Clifford symmetric, odd resp.\ even, regular self-adjoint Fredholm operator on a $\Z_2$-graded Hilbert $A$-module $E$, then $[\D] = 0$ in $\KK^0(\C,A)$ resp.\ $\KK^1(\C,A)$. 
\end{lem}
\begin{proof}
Since any normalising function $\chi$ is odd, we have (cf.\ \cite[Lemma 1.14]{DM20}) $\gamma \chi(\D) = -\chi(\D) \gamma$. 
Then $F_t := \cos(\tfrac{\pi t}{2}) \chi(\D) + \sin(\tfrac{\pi t}{2}) \gamma$ yields an operator-homotopy from $\chi(\D)$ to $\gamma$, and since $\gamma$ is degenerate we conclude $[\D] = [\chi(\D)] = [\gamma] = 0 \in \KK^\bullet(\C,A)$. 
\end{proof}

The following statement generalises \cite[Lemma 4.6]{vdD19_Index_DS} to $\KK^1$. 
It shows that if $\D$ is `almost Clifford symmetric', then it is equivalent to a Clifford symmetric operator. 

\begin{prop}
\label{prop:Fred_almost-Cliff-symm}
Let $\D$ be an odd resp.\ even, regular self-adjoint Fredholm operator on a $\Z_2$-graded Hilbert $A$-module $E$. 
Suppose there exists a self-adjoint unitary $\gamma\in\mL_A(E)$ such that $\gamma \cdot \Dom\D \subset \Dom\D$ and $\gamma\D+\D\gamma$ is relatively $\D$-compact. 
If $\D$ is odd, then we require in addition that $\gamma$ is also odd. 
Then $[\D] = 0$ in $\KK^0(\C,A)$ resp.\ $\KK^1(\C,A)$. 
\end{prop}
\begin{proof}
Since $\D + \gamma\D\gamma$ is symmetric and relatively $\D$-compact, we know from \cref{prop:KK-class_rel_cpt_pert} that $[\D] = [\D']$, where $\D' := \frac12(\D-\gamma\D\gamma) = \D - \frac12(\D + \gamma\D\gamma)$. Since $\D'$ has the Clifford symmetry $\gamma$, it follows from \cref{lem:Fred_Cliff-symm} that $[\D] = [\D'] = 0$. 
\end{proof}

\subsection{The (odd) relative index}
\label{sec:rel-ind}

Consider two projections $P,Q \in \mL_A(E)$, such that the difference $P-Q$ is a \emph{compact} endomorphism on $E$. 
We will consider a relative index both in the graded and in the ungraded case. 
We briefly recall the definitions and some useful properties; for more details, we refer to \cite[\S3.2 \& \S8.1]{Wah07}. 

We will use $p=0,1$ to distinguish between the ungraded case $p=0$ and the graded case $p=1$. 
We first consider the \emph{ungraded} case $p=0$, as in \cite[\S2.2.1]{vdD25_Callias}, following \cite[\S3.2]{Wah07}. 
Since $P-Q$ is compact, the operator $Q \colon \Ran(P) \to \Ran(Q)$ is a Fredholm operator (with parametrix $P \colon \Ran(Q) \to \Ran(P)$) and thus has a $\K_0(A)$-valued index. 
The index of such a Fredholm operator is defined in \S\ref{sec:Fredholm_KK}. 

Next, in the \emph{graded} case $p=1$, we assume $E = E^+ \oplus E^-$ is $\Z_2$-graded, and we require in addition that $2P-1$ and $2Q-1$ are odd. 
It then follows that $E^+$ and $E^-$ are unitarily isomorphic. 
More precisely, since $2P-1$ is an odd self-adjoint unitary, we can write 
\begin{equation}
\label{eq:proj_odd}
2P-1 = \mattwo{0}{U_P^*}{U_P}{0} \quad \text{or equivalently} \quad P = \frac12 \mattwo{1}{U_P^*}{U_P}{1} ,
\end{equation}
where $U_P\colon E^+ \to E^-$ is unitary. 
Since $P-Q$ is compact, also $U_P-U_Q$ is compact. 
It follows that $1 - U_P U_Q^*$ is compact on $E^-$, and therefore $U_P U_Q^*$ lies in the minimal unitisation of the compact endomorphisms on $E^-$. 

We recall that for any unitary operator $V$ in the minimal unitisation of $\mK_A(E)$, we can use Kasparov stabilisation $E \oplus \mH_A \simeq \mH_A$ to obtain a class $[V] \in K_1(A)$ defined by 
\[
[V] := [V\oplus1] \in K_1(\mK_A(E\oplus\mH_A)) \simeq K_1(\mK_A(\mH_A)) \simeq K_1(A) .
\]

\begin{defn}
\label{defn:rel-ind}
Consider projections $P,Q \in \mL_A(E)$ with $P-Q \in \mK_A(E)$. 
We define the (even or odd) relative index $\relind_p(P,Q) \in \K_p(A)$ as follows. 
In the ungraded case $p=0$, we define the \emph{(even) relative index of $(P,Q)$} by 
\[
\relind_0(P,Q) := \Index \big( Q \colon \Ran(P) \to \Ran(Q) \big) \in \K_0(A) .
\]
In the graded case $p=1$, we additionally require $2P-1$ and $2Q-1$ to be odd, and define the \emph{odd relative index of $(P,Q)$} \cite[\S8.1]{Wah07} by 
\[
\relind_1(P,Q) := \left[ \mattwo{1}{0}{0}{U_P U_Q^*} \right] \in \K_1(A) ,
\]
where $U_P$ and $U_Q$ are obtained from $P$ and $Q$ as in \cref{eq:proj_odd}.
\end{defn}

Here we denote the (even or odd) relative index $\relind_p$ with a subscript $p$, to remind us that it takes values in $\K_p(A)$. 
We record two important properties of the relative index:
\begin{lem}[{\cite[\S3.2 \& \S8.1]{Wah07}}]
\label{lem:rel-ind_properties}
\begin{itemize}
\item (Additivity.) 
Let $P,Q,R \in \mL_A(E)$ be projections with $P-Q$ and $Q-R$ compact. 
(If $p=1$, we assume $2P-1$, $2Q-1$, and $2R-1$ to be odd.) 
Then 
\[
\relind_p(P,R) = \relind_p(P,Q) + \relind_p(Q,R) .
\]
\item (Homotopy invariance.) 
Let $\{P_t\}_{t\in[0,1]}$ and $\{Q_t\}_{t\in[0,1]}$ be strongly continuous paths of projections such that $P_t-Q_t$ is compact for each $t\in[0,1]$. 
(If $p=1$, we assume $2P_t-1$ and $2Q_t-1$ to be odd.) 
Then 
\[
\relind_p(P_0,Q_0) = \relind_p(P_1,Q_1) .
\]
\end{itemize}
\end{lem}

For future convenience we also record the following simple result. 
\begin{lem}
\label{lem:relind_PUP}
Consider a projection $P \in \mL_A(E)$ and a unitary $U \in \mL_A(E)$ with $[P,U] \in \mK_A(E)$. 
Then $\relind_0(P,U^*PU) = \Index(PUP)$. 
\end{lem}
\begin{proof}
We note that the assumption implies that $P-U^*PU$ is compact (so that the relative index is well-defined) and that $PUP$ is Fredholm (with parametrix $PU^*P$) on $\Ran P$. We then compute
\begin{align*}
\relind_0(P,U^*PU) 
&= \Index \big( U^*PUP \colon \Ran(P) \to \Ran(U^*PU) \big) \\
&= \Index \big( PUP \colon \Ran(P) \to \Ran(PU) \big) \\
&= \Index \big( PUP \colon \Ran(P) \to \Ran(P) \big) .
\qedhere 
\end{align*}
\end{proof}

\subsection{The (odd) spectral flow}

We follow the conventions and definitions for the spectral flow as given in \cite[\S2.2]{vdD25_Callias} (see also references therein, in particular \cite[\S3]{Wah07}). 
Here we shall only briefly adapt the definitions to the graded case, following the approach of \cite[\S8]{Wah07}. 

In the graded case $p=1$, we assume $E$ is $\Z_2$-graded, and we consider \emph{odd} operators $\D$ on $E$ or \emph{odd} families of operators $\{\D(x)\}_{x\in X}$ on $E$, defining \emph{odd} operators $\D(\cdot)$ on $C(X,E)$ (where $X$ is a compact Hausdorff space). 
In this case, we require trivialising operators and (locally) trivialising families to be \emph{odd} as well. 
The following lemma shows that the $\Z_2$-grading is compatible with the positive spectral projection. 
\begin{lem}
\label{lem:pos-proj_odd}
Let $\D$ be an odd invertible regular self-adjoint operator on a $\Z_2$-graded Hilbert $A$-module $E$, and let $P_+(\D)$ denote the positive spectral projection of $\D$. 
Then the operator $2P_+(\D)-1$ is also odd. 
\end{lem}
\begin{proof}
Since $\D$ is invertible, there exists $\epsilon>0$ such that $(-\epsilon,\epsilon)$ does not intersect the spectrum of $\D$. 
We can then write $2P_+(\D)-1 = \phi(\D)$ for any odd function $\phi\in C_b(\R)$ with $\phi(r) = -1$ for all $r\leq-\epsilon$ and $\phi(r)=1$ for all $r\geq\epsilon$. 
The statement then follows because $\phi(\D)$ is odd for any odd function $\phi$ (cf.\ \cite[Lemma 1.14]{DM20}). 
\end{proof}

Now let $\D$ be a regular self-adjoint Fredholm operator on $E$, and let $\B_0$ and $\B_1$ be two trivialising operators for $\D$. 
We note that our trivialising operators, defined as in \cite[Definition 2.5]{vdD25_Callias}, are required to be relatively compact but not necessarily bounded (in contrast with \cite[Definition 3.4]{Wah07}). 
We consider the ungraded case ($p=0$) and the graded case ($p=1$) simultaneously (if $p=1$, $E$ is $\Z_2$-graded and $\D,\B_0,\B_1$ are odd). 
By \cite[Corollary A.10]{vdD25_Callias}, the difference of spectral projections $P_+(\D+\B_1) - P_+(\D+\B_0)$ is compact (if $p=1$, the operators $2P_+(\D+\B_j)-1$ are odd by \cref{lem:pos-proj_odd}). 
Hence we can define
\begin{equation}
\label{eq:ind}
\ind_p(\D,\B_0,\B_1) := \relind_p\big( P_+(\D+\B_1) , P_+(\D+\B_0) \big) .
\end{equation}

\begin{defn}[{cf.\ \cite[Definition 3.10]{Wah07} \& \cite[Definition 2.7]{vdD25_Callias}}]
\label{defn:spectral_flow}
Let $p=0$ or $p=1$. 
Let $\D(\cdot) = \{\D(t)\}_{t\in[0,1]}$ be a regular self-adjoint operator on the Hilbert $C([0,1],A)$-module $C([0,1],E)$, for which locally trivialising families exist. (If $p=1$, we require $E$ to be $\Z_2$-graded and $\D$ (as well as all trivialising families) to be odd.) 
Let $0 = t_0 < t_1 < \ldots < t_n = 1$ be such that there is a trivialising family $\{\B^i(t)\}_{t\in[t_i,t_{i+1}]}$ of $\{\D(t)\}_{t\in[t_i,t_{i+1}]}$ for each $i=0,\ldots,n-1$. 
Let $\A_0$ and $\A_1$ be trivialising operators of $\D(0)$ and $\D(1)$. 
Then we define 
\begin{align*}
&\SF_p\big(\{\D(t)\}_{t\in[0,1]} ; \A_0,\A_1 \big) 
:= \\
&\qquad \ind_p\big(\D(0),\A_0,\B^0(0)\big) + \sum_{i=1}^{n-1} \ind_p\big(\D(t_i),\B^{i-1}(t_i),\B^i(t_i)\big) + \ind_p\big(\D(1),\B^{n-1}(1),\A_1\big) \\
&\qquad \in K_p(A) ,
\end{align*}
where $\ind_p$ is defined in \cref{eq:ind}. 

If the endpoints $\D(0)$ and $\D(1)$ are invertible, we define 
\begin{multline*}
\SF_p\big(\{\D(t)\}_{t\in[0,1]} \big) 
:= \SF_p\big(\{\D(t)\}_{t\in[0,1]} ; 0,0 \big) \\
= \ind_p\big(\D(0),0,\B^0(0)\big) + \sum_{i=1}^{n-1} \ind_p\big(\D(t_i),\B^{i-1}(t_i),\B^i(t_i)\big) + \ind_p\big(\D(1),\B^{n-1}(1),0\big) .
\end{multline*}
We call $\SF_0$ the (even) \emph{spectral flow} and $\SF_1$ the \emph{odd spectral flow} of the family $\{\D(t)\}_{t\in[0,1]}$. 
\end{defn}

As in \cite{Wah07}, the definition of the spectral flow is independent of the choice of subdivision and the choice of trivialising families $\{\B^i(t)\}_{t\in[t_i,t_{i+1}]}$. 
In particular, using \cite[Lemma 3.5]{Wah07}, we may choose the trivialising families to be \emph{bounded}, and thus we recover the definition of the (even) spectral flow given in \cite[Definition 3.10]{Wah07} or the odd spectral flow in \cite[\S8.1]{Wah07}.

\subsubsection{Relatively compact perturbations}
\label{sec:sf_rel-cpt}

We can adapt \cite[Proposition 2.8]{vdD25_Callias} to the graded case $p=1$. 
\begin{prop}
\label{prop:sf_rel_cpt_family}
Let $\D(\cdot) = \{\D(t)\}_{t\in[0,1]}$ be a regular self-adjoint operator on the Hilbert $C([0,1],A)$-module $C([0,1],E)$. 
(In the graded case $p=1$, $E$ is $\Z_2$-graded and $\D$ is odd.) 
Assume that
\begin{itemize}
\item 
the endpoints $\D(0)$ and $\D(1)$ are invertible; 
\item 
$\D(t) \colon \Dom\D(0) \to E$ depends norm-continuously on $t$; and
\item 
$\D(t)-\D(0)$ is relatively $\D(0)$-compact for each $t\in[0,1]$. 
\end{itemize}
Then the following statements hold:
\begin{enumerate}
\item 
There exists a trivialising family for $\{\D(t)\}_{t\in[0,1]}$. 
\item 
We have the equality 
\begin{align}
\label{eq:sf_glob}
\SF_p\big(\{\D(t)\}_{t\in[0,1]} \big) 
&= \relind_p\big( P_+(\D(1)) , P_+(\D(0)) \big) .
\end{align}
\end{enumerate}
\end{prop}
\begin{proof}
The statement is proven in \cite[Proposition 2.8]{vdD25_Callias} for the ungraded case $p=0$, but the same proof works also in the graded case $p=1$. 
\end{proof}

As a consequence of the above proposition, we obtain an abstract generalisation of the classical ``desuspension'' result by Booss-Bavnbek and Wojciechowski \cite[Theorem 17.17]{Booss-Bavnbek--Wojciechowski93}. We note that, as a special case, we also obtain a ``family version'' (cf.\ \cite[Theorem 4.4]{DZ98}) by considering $A=C(X)$ for a compact space $X$. 
\begin{coro}
\label{coro:sf_index_PuP}
Let $\D$ be an invertible regular self-adjoint operator on the Hilbert $A$-module $E$. 
Consider a unitary operator $u \in \mL_A(E)$ such that $u\colon\Dom\D\to\Dom\D$ and $[\D,u]$ is relatively $\D$-compact. 
Let $\chi \colon [0,1] \to \R$ be any continuous function satisfying $\chi(0) = 0$ and $\chi(1) = 1$. 
For $t\in[0,1]$ we define $\D(t) := (1-\chi(t)) \D + \chi(t) u^* \D u = \D + \chi(t) u^* [\D,u]$. 
Then we have the equality 
\[
\SF_0\big(\{\D(t)\}_{t\in[0,1]} \big) 
= - \Index \big( P_+(\D) u P_+(\D) \big) 
\in \K_0(A) .
\]
\end{coro}
\begin{proof}
From \cref{prop:sf_rel_cpt_family} we have 
\[
\SF_0\big(\{\D(t)\}_{t\in[0,1]} \big) 
= \relind_0\big( P_+(\D(1)) , P_+(\D(0)) \big) 
= \relind_0\big( P_+(u^*\D u) , P_+(\D) \big) .
\]
Using \cref{lem:relind_PUP} we thus obtain 
\[
\relind_0\big( P_+(u^* \D u) , P_+(\D) \big) 
= \relind_0\big( u^* P_+(\D) u , P_+(\D) \big)
= -\Index\big( P_+(\D) u P_+(\D) \big) .
\qedhere 
\]
\end{proof}

\subsubsection{Bott periodicity}

The following result shows that the (even or odd) spectral flow is closely related to the Bott periodicity isomorphism $\beta\colon\K_{p+1}(C_0(\R,A)) \to \K_p(A)$. 
Let $\D(\cdot) = \{\D(t)\}_{t\in[0,1]}$ be a regular self-adjoint Fredholm operator on the Hilbert $C([0,1],A)$-module $C([0,1],E)$, with $\D(0)$ and $\D(1)$ invertible. 
(In the graded case $p=1$, $E$ is $\Z_2$-graded and each $\D(t)$ is odd.) 
We extend the family to $\R$ by setting $\D(t) := \D(0)$ for all $t<0$ and $\D(t) := \D(1)$ for all $t>1$, and we view $\D(\cdot)$ as a regular self-adjoint Fredholm operator on the Hilbert $C_0(\R,A)$-module $C_0(\R,E)$, defining a class $[\D(\cdot)] \in \KK^{p+1}(\C,C_0(\R,A))$. 
We note that the Bott periodicity isomorphism is implemented by taking the Kasparov product with the class $[-i\partial_t] \in \KK^1(C_0(\R),\C)$: 
\[
\beta\big([\D(\cdot)]\big) = [\D(\cdot)] \otimes_{C_0(\R)} [-i\partial_t] \in K_p(A) .
\]
The following result shows that we can view the (even or odd) spectral flow $\SF_p$ as implementing the Bott periodicity isomorphism $\KK^{p+1}(\C,C_0(\R,A)) \to \KK^p(\C,A) \simeq \K_p(A)$ (whenever it is defined). 
\begin{prop}[\cite{Wah07}]
\label{prop:SF_Kasp_prod}
If there exist locally trivialising families for $\{\D(t)\}_{t\in\R}$, then
\[
\SF_p\big(\{\D(t)\}_{t\in[0,1]}\big) = \beta\big([\D(\cdot)]\big) = [\D(\cdot)] \otimes_{C_0(\R)} [-i\partial_t] \in K_p(A) .
\]
\end{prop}
\begin{proof}
We obtain the equality $\beta\big([\D(\cdot)]\big) = \SF_p\big(\{\D(x)\}_{x\in\R}\big)$ for $p=0$ from \cite[Proposition 4.2 \& Theorem 4.4]{Wah07} and for $p=1$ from \cite[Theorem 8.6]{Wah07}. 
\end{proof}

\section{Computation of `higher' index pairings}
\label{sec:index_pairings}

We will consider the pairing (via the Kasparov product) of $\K$-theory with $\K$-homology. 
The two promiment and well-known cases are the \emph{even index pairing} (pairing even $\K$-theory with even $\K$-homology) and the \emph{odd index pairing} (pairing odd $\K$-theory with odd $\K$-homology), see e.g.\ \cite[Ch.\ IV, Proposition 2]{Connes94}. 
Given an even Fredholm module $(B,\mH,F)$ and a projection $p\in M_n(B)$, 
the even index pairing of the $\K$-theory class $[p] \in \K_0(B)$ with the $\K$-homology class $[F] \in \K^0(B)$ is given by 
\[
[p] \otimes_B [F] \; \xmapsto{\simeq} \; \Index( p F_+ p ) ,
\]
where $F_+$ is one of the $\Z_2$-graded components of $F$. 
Given an odd Fredholm module $(B,\mH,F)$ and a unitary $u\in M_n(B)$, 
the odd index pairing of the $\K$-theory class $[u] \in \K_1(B)$ with the $\K$-homology class $[F] \in \K^1(B)$ is given by 
\[
[u] \otimes_B [F] \; \xmapsto{\simeq} \; \Index( P_+ u P_+ ) , 
\]
where $P_+ \equiv P_+(F)$ denotes the positive spectral projection of $F$. 

Here we will reprove these results in a slightly more general setting: we consider $\K$-theory `with coefficients' in an auxiliary $C^*$-algebra $A$. 
Thus we consider an element in the $\K$-theory group $\K_*(A\otimes B)$ represented by a projection $p \in M_n(A\otimes B)$ or a unitary $u \in M_n(A\otimes B)$, yielding `higher' index pairings taking values in $\K_*(A)$. 
Moreover, we will also consider the pairings of odd $\K$-theory with even $\K$-homology and of even $\K$-theory with odd $\K$-homology. 
These pairings are normally ignored in the literature, for the simple reason that (for $A=\C$) the odd $\K$-theory group $\K_1(\C) = \{0\}$ is trivial. 
However, for a general $C^*$-algebra $A$, the group $\K_1(A)$ can be nontrivial, and then the nontrivial pairings $\K_p(A\otimes B) \times \K^{p+1}(B) \to \K_1(A)$ can be of interest as well. 

Finally, we will consider here the unbounded picture of $\K$-homology. Thus, an element of the $\K$-homology group $\K^p(B)$ is represented by an even (if $p=0$) or odd (if $p=1$) \emph{spectral triple} $(\B,\mH,\D)$. 
We will denote the representation of $B$ on $\mH$ by $\rho \colon B \to \B(\mH)$.

\subsection{The even index pairing}
\label{sec:even_index_pairing}

Given a projection $p\in M_n(A\otimes B)$, we recall that its $\K$-theory class $[p] \in \K_0(A\otimes B)$ corresponds (under the standard isomorphism) to the $\KK$-class of the (unbounded) Kasparov $\C$-$A\otimes B$ module $(\C,p\cdot (A\otimes B)^{\oplus n},0)$ in $\KK^0(\C,A\otimes B)$. 
Given an even spectral triple $(\B,\mH,\D)$ representing a class $[\D] \in \KK^0(B,\C)$, there is a $\Z_2$-grading $\mH = \mH_+ \oplus \mH_-$, with respect to which we can write 
\begin{align*}
\D &= \mattwo{0}{\D_-}{\D_+}{0} , & 
\rho(b) &= \mattwo{\rho_+(b)}{0}{0}{\rho_-(b)} , \quad\forall b\in B .
\end{align*}
The projection $p \in M_n(A\otimes B)$ yields a projection $\rho(p) \in A\otimes\B(\mH^{\oplus n})$ by extending the representation $\rho = \rho_+ \oplus \rho_-$ from $\mH$ to $\mH^{\oplus n}$. Similarly, we also obtain projections $p_+ \equiv \rho_+(p) \in A\otimes\B(\mH_+^{\oplus n})$ and $p_- \equiv \rho_-(p) \in A\otimes\B(\mH_-^{\oplus n})$. 

In addition to the even index pairing $\K_0(A\otimes B) \times \KK^0(B,\C) \to \K_0(A)$, we will simultaneously also consider the even-odd pairing $\K_0(A\otimes B) \times \KK^1(B,\C) \to \K_1(A)$. 
\begin{thm}
\label{thm:even_index_pairing_unbdd}
Let $A$ and $B$ be trivially graded unital $C^*$-algebras. 
Consider an (even or odd) spectral triple $(\B,\mH,\D)$ over $B$, representing a class $[\D] \in \KK^p(B,\C)$. 
Furthermore, consider a projection $p \in M_n(A\otimes B)$, such that $\rho(p)$ preserves the domain of $\D$, and $[\D,\rho(p)]$ is bounded. 
Then the pairing $[p] \otimes_B [\D] \in \KK^p(\C,A)$ is given by 
\[
[p] \otimes_B [\D] = \big[ \rho(p) \D \rho(p) \big] \in \KK^p(\C,A) .
\]
If $p=0$, the pairing is given under the standard isomorphism $\KK^0(\C,A) \xrightarrow{\simeq} \K_0(A)$ by 
\[
[p] \otimes_B [\D] \; \xmapsto{\simeq} \; \Index\big( p_- \D_+ p_+ \big) \in \K_0(A) .
\]
\end{thm}
\begin{proof}
We have the isomorphism 
\[
(p\cdot (A\otimes B)^{\oplus n}) \otimes_B \mH 
\simeq \rho(p) \cdot (A\otimes \mH)^{\oplus n} . 
\]
Consider first the case of an \emph{even} spectral triple ($p=0$). 
We will compute the Kasparov product using \cite[Theorem 7.4]{LM19}. 
We consider $T := \rho(p) \D \rho(p)$ on $\rho(p) \cdot (A\otimes \mH)^{\oplus n}$ (where we simply write $\D$ for the diagonal operator on $(A\otimes \mH)^{\oplus n}$), which satisfies condition (i) of \cite[Theorem 7.4]{LM19}. 
Since the Kasparov module $(\C,p\cdot (A\otimes B)^{\oplus n},0)$ representing the $\K$-theory class $[p]$ trivially satisfies conditions (ii) and (iii) of \cite[Theorem 7.4]{LM19}, 
it follows that the unbounded Kasparov module $\big(\C,\rho(p) \cdot (A\otimes \mH)^{\oplus n},\rho(p) \D \rho(p)\big)$ represents the Kasparov product $[p] \otimes_B [\D]$. 
With respect to the $\Z_2$-grading $\mH = \mH_+ \oplus \mH_-$ we write 
\begin{align*}
\rho(p) \D \rho(p) &= \mattwo{p_+}{0}{0}{p_-} \mattwo{0}{\D_-}{\D_+}{0} \mattwo{p_+}{0}{0}{p_-} = \mattwo{0}{p_+\D_-p_-}{p_-\D_+p_+}{0} .
\end{align*}
Finally, under the standard isomorphism, the $\KK$-class $[\rho(p)\D\rho(p)] \in \KK^0(\C,A)$ corresponds to $\Index(p_-\D_+p_+) \in \K_0(A)$. 

For the case of an \emph{odd} spectral triple ($p=1$), using the description of the even-odd Kasparov product from \cite[Example 2.37]{BMS16}, the pairing $[p] \otimes_B [\D]$ is now represented by $\Gamma \rho(p) \D \rho(p)$, where $\Gamma$ denotes the $\Z_2$-grading of the Kasparov module $(\C,p\cdot (A\otimes B)^{\oplus n},0)$. But the latter is trivially graded, so $\Gamma=\Id$. 
\end{proof}

\begin{remark}[The nonunital case]
\label{remark:nonunital_index_pairing}
While we focus on the case of unital $C^*$-algebras, we also briefly discuss here the nonunital case (a more detailed discussion can be found in e.g.\ \cite[\S2.3]{CGRS14}). 
So suppose that we consider a spectral triple $(\B,\mH,\D)$ over a \emph{nonunital} $C^*$-algebra $B$. 
We will assume that $\D$ is \emph{invertible} (in general, we may replace $\D$ by an invertible operator using a doubling trick due to Connes \cite[\S{}I.6]{Con85}). 
In this case, $\rho(p) \D \rho(p)$ is Fredholm with parametrix $\rho(p) \D^{-1} \rho(p)$. 
We wish to compute the pairing with an element $[p]-[q] \in \K_0(A\otimes B)$, where $p,q \in M_n(A\otimes B^+)$ are projections such that $p-q \in M_n(A\otimes B)$. 
In the case of an \emph{even} spectral triple, we obtain an additional summand $\rho(q) \D \rho(q)$ equipped with the \emph{opposite} $\Z_2$-grading, and the index pairing is now given by 
\begin{align*}
\big( [p] - [q] \big) \otimes_B [\D] &= 
\big[ \rho(p) \D \rho(p) \oplus \rho(q) \D \rho(q) \big] \\
&= \Index\big( p_- \D_+ p_+ \big) + \Index\big( q_+ \D_- q_- \big) \in \K_0(A) .
\end{align*}
In the case of an \emph{odd} spectral triple, the additional summand $\rho(q) \D \rho(q)$ has a \emph{minus sign} (coming from the $\Z_2$-grading of the $\K$-theory class), and the even-odd pairing is given by 
\begin{align*}
\big( [p] - [q] \big) \otimes_B [\D] &= \big[ \rho(p) \D \rho(p) \oplus \big({- \rho(q)} \D \rho(q)\big) \big] \\
&= \big[ \rho(p) \D \rho(p) \big] - \big[ \rho(q) \D \rho(q) \big] 
\in \KK^1(\C,A) .
\end{align*}
\end{remark}

\subsection{The odd index pairing}
\label{sec:odd_index_pairing}

Consider a unitary $u\in M_n(A\otimes B)$ and its odd $\K$-theory class $[u] \in \K_1(A\otimes B)$. 
Given an even or odd spectral triple $(\B,\mH,\D)$ representing a class $[\D] \in \KK^p(B,\C)$, our aim is to compute the pairing $[u] \otimes_B [\D] \in \K_{p+1}(B)$. We follow similar arguments as in the proof of \cite[Proposition A.6]{Wah09}, which describes a similar pairing $\K_1(B) \times \KK_*(B,A) \to \K_{*+1}(A)$. The first step consists of replacing the unitary $u$ by a projection $P_u$. 
For this purpose, we consider smooth functions $\chi_\pm \colon \R \to [0,1]$ with the following properties:
\begin{itemize}
\item $\supp \chi_+ \subset [\frac12,\infty)$ and $\supp \chi_- \subset (-\infty,\frac12]$, and 
\item $\chi_+(t) = 1$ for $t\geq1$ and $\chi_-(t) = 1$ for $t\leq0$.
\end{itemize}
We define a third function $\chi_0 \in C_c^\infty(\R)$ by $\chi_0 := \sqrt{1-\chi_+^2-\chi_-^2}$. 

\begin{lem}
\label{lem:proj_P_u}
Let $A$ be a (trivially graded) unital $C^*$-algebra, and let $u\in M_n(A)$ be unitary. Define two projections in $M_{2n}\big(C_0(\R,A)^+\big)$ given by 
\[
P_u := \mattwo{\chi_-^2+\chi_+^2}{\chi_0(\chi_-+\chi_+ u)}{\chi_0(\chi_-+\chi_+ u^*)}{\chi_0^2} ,
\qquad 
p_n := \mattwo{1}{0}{0}{0} .
\]
Then we have the equality 
\[
[u] = \big( [P_u]-[p_n] \big) \otimes_{C_0(\R)} [-i\partial_t] \in \K_1(A) .
\]
\end{lem}
\begin{proof}
Consider for $t\in\R$ the unitary 
\[
w(t) := \mattwo{\chi_-(t)+\chi_+(t) u}{-\chi_0(t)}{\chi_0(t)}{\chi_-(t)+\chi_+(t) u^*} \in M_{2n}(A) ,
\]
and note that 
\[
w(-t) = \mattwo{1}{0}{0}{1} 
\quad\text{and}\quad 
w(t) = \mattwo{u}{0}{0}{u^*} 
\quad\text{for all } t\geq1 .
\]
An explicit computation shows that 
\[
P_u(t) = w(t) p_n w(t)^* . 
\]
Thus, by \cite[Theorem 8.2.2]{Blackadar98}, the image of $[u]$ under the natural isomorphism $\K_1(A) \to \K_0(C_0(\R,A))$ is given by 
\[
[P_u] - [p_n] \in \K_0(C_0(\R,A)) .
\]
The inverse isomorphism $\K_0(C_0(\R,A)) \to \K_1(A)$ can be implemented by taking the interior Kasparov product (over $C_0(\R)$) with $[-i\partial_t]$. 
For the sake of completeness, we compute this Kasparov product explicitly. 
We represent $[P_u]-[p_n]$ by the Kasparov $\C$-$C_0(\R,A)$-module 
\[
\left( \C , P_u A^{\oplus 2n} \oplus p_n A^{\oplus 2n} , \mattwo{0}{P_u}{p_n}{0} \right) .
\]
Under the unitary isomorphism $w^* \colon P_u A^{\oplus 2n} \xrightarrow{\simeq} p_n A^{\oplus 2n} = A^{\oplus n}$, the above Kasparov module is unitarily equivalent to $(\C , A^{\oplus n} \oplus A^{\oplus n} , T_u)$, where 
\[
T_u = \mattwo{0}{p_n w^* p_n}{p_n w p_n}{0} = \mattwo{0}{\chi_- + \chi_+ u^*}{\chi_- + \chi_+ u}{0} .
\]
From \cref{prop:SF_Kasp_prod,prop:sf_rel_cpt_family} we obtain 
\begin{align*}
\big( [P_u]-[p_n] \big) \otimes_{C_0(\R)} [-i\partial_t]
&= [T_u] \otimes_{C_0(\R)} [-i\partial_t] 
= \SF_1\big( \{T_u(t)\}_{t\in[0,1]} \big) \\
&= \relind_1 \big( P_+(T_u(1)) , P_+(T_u(0)) \big) \\
&= \relind_1 \left( \frac12 \mattwo{1}{u^*}{u}{1} , \frac12 \mattwo{1}{1}{1}{1} \right) 
= [u] .
\qedhere
\end{align*}
\end{proof}

Now consider a unitary $u\in M_n(A\otimes B)$ and let $(\B,\mH,\D)$ be an even or odd spectral triple. 
In addition to the odd index pairing $\K_1(A\otimes B) \times \KK^1(B,\C) \to \K_0(A)$, we will simultaneously also consider the odd-even pairing $\K_1(A\otimes B) \times \KK^0(B,\C) \to \K_1(A)$. 
It will be convenient to take $\D$ to be invertible; we will show afterwards that, at least in the odd case, this invertibility assumption can be removed. 
We recall that the representation $\rho\colon B\to\B(\mH)$ is \emph{nondegenerate} if $\Span\{ \rho(b) h \mid b\in B, h\in\mH \}$ is dense in $\mH$. 

\begin{thm}
\label{thm:odd_index_pairing}
Let $A$ and $B$ be trivially graded unital $C^*$-algebras. 
Consider an (even or odd) spectral triple $(\B,\mH,\D)$ over $B$, representing a class $[\D] \in \KK^p(B,\C)$. 
We assume that the representation $\rho \colon B \to \B(\mH)$ is nondegenerate, and that $\D$ is invertible. 
Furthermore, consider a unitary $u \in M_n(A\otimes B)$, such that $\rho(u)$ preserves the domain of $\D$, and $[\D,\rho(u)]$ is bounded. 
Then the pairing $[u] \otimes_B [\D] \in \KK^{p+1}(\C,A)$ is given by 
\[
[u] \otimes_B [\D] = \relind_{p+1}\big( P_+(\D) , \rho(u^*) P_+(\D) \rho(u) \big) \in \KK^{p+1}(\C,A) .
\]
If $p=1$, the pairing is given under the standard isomorphism $\KK^0(\C,A) \xrightarrow{\simeq} \K_0(A)$ by 
\[
[u] \otimes_B [\D] = \Index\big( P_+(\D) \rho(u) P_+(\D) \big) \in \K_0(A) .
\]
\end{thm}
\begin{proof}
Applying \cref{lem:proj_P_u} (with the $C^*$-algebra $A\otimes B$) and using associativity of the Kasparov product, we compute 
\begin{align*}
[u] \otimes_B [\D] 
&= \Big( \big( [P_u]-[p_n] \big) \otimes_{C_0(\R)} [-i\partial_t] \Big) \otimes_B [\D] \\
&= \big( [P_u]-[p_n] \big) \otimes_{C_0(\R,B)} \big( [-i\partial_t] \otimes [\D] \big) \\
&= - \big( [P_u]-[p_n] \big) \otimes_{C_0(\R,B)} \big( [\D] \otimes [-i\partial_t] \big) \\
&= - \big( [P_u]-[p_n] \big) \otimes_{C_0(\R,B)} \tau_{C_0(\R)}([\D]) \otimes_{C_0(\R)} [-i\partial_t] .
\end{align*}
where in the third step we used that the exterior Kasparov product of odd Kasparov modules is anticommutative. 
Computing the nonunital pairing described in \cref{remark:nonunital_index_pairing}, 
we obtain (suppressing the representation $\rho$ from our notation) 
\[
\big( [P_u]-[p_n] \big) \otimes_{C_0(\R,B)} \tau_{C_0(\R)}([\D]) 
= \left[ P_u \D P_u \right] ,
\]
where the summand $p_n \D p_n$ disappears because this operator is invertible. 
Recalling that $P_u$ is defined via conjugation with $w(t)$, we first explicitly compute the upper left entry of 
\[
w(t)^* \D w(t) = \mattwo{(\chi_-^2+\chi_0^2)\D + \chi_+^2 u^* \D u}{\ldots}{\ldots}{\ldots} .
\]
Since the $\K$-theory class of a Fredholm operator is stable under unitary transformations, we obtain 
\begin{align}
\label{eq:P_u-D-P_u}
\left[ P_u \D P_u \right] 
&= \left[ p_n w^* \D w p_n \right] 
= \left[ (\chi_-^2+\chi_0^2)\D + \chi_+^2 u^* \D u \right] 
= \left[ \D(\cdot) \right] ,
\end{align}
where we introduce the family $\{\D(t)\}_{t\in\R}$ given by 
\[
\D(t) := (1-\chi_+(t)^2)\D + \chi_+(t)^2 u^* \D u = \D + \chi_+(t)^2 u^* [\D,u] .
\]
We note that $\D(t) = \D$ for all $t\leq0$ and $\D(t) = u^* \D u$ for all $t\geq1$. 
The Kasparov product $[\D(\cdot)] \otimes_{C_0(\R)} [-i\partial_t]$ can be identified with the (even or odd) spectral flow of $\{\D(t)\}$ by \cref{prop:SF_Kasp_prod}: 
\[
[\D(\cdot)] \otimes_{C_0(\R)} [-i\partial_t] 
= \SF_{p+1}\big( \{\D(t)\}_{t\in\R} \big) .
\]
Using \cref{prop:sf_rel_cpt_family} (note that $u^*[\D,u]$ is bounded and therefore relatively $\D$-compact), this spectral flow is given by the relative index of the positive spectral projections of $\D(0)$ and $\D(1)$:
\begin{align*}
\SF_{p+1}\big(\{\D(t)\}_{t\in[0,1]} \big) 
&= \relind_{p+1}\big( P_+(\D(1)) , P_+(\D(0)) \big) \\
&= - \relind_{p+1}\big( P_+(\D) , u^* P_+(\D) u \big) .
\end{align*}
To summarise, we have proven the sequence of equalities 
\begin{align*}
[u] \otimes_B [\D] 
&= - \big( [P_u]-[p_n] \big) \otimes_{C_0(\R,B)} \tau_{C_0(\R)}([\D]) \otimes_{C_0(\R)} [-i\partial_t] \\
&= - \left[ P_u \D P_u \right] \otimes_{C_0(\R)} [-i\partial_t] 
= - [\D(\cdot)] \otimes_{C_0(\R)} [-i\partial_t] \\
&= - \SF_{p+1}\big( \{\D(t)\}_{t\in[0,1]} \big) 
= \relind_{p+1}\big( P_+(\D) , u^* P_+(\D) u \big) . 
\end{align*}
In the special case $p=1$, we use \cref{lem:relind_PUP} to rewrite the (even) relative index as an ordinary Fredholm index:
\[
\relind_0 \big( P_+(\D) , u^* P_+(\D) u \big) 
= \Index \big( P_+(\D) u P_+(\D) \big) . 
\qedhere
\]
\end{proof}

\begin{coro}
\label{coro:odd_index_pairing}
Let $A$ and $B$ be trivially graded unital $C^*$-algebras. 
Consider an \emph{odd} spectral triple $(\B,\mH,\D)$ over $B$, representing a class $[\D] \in \KK^1(B,\C)$. 
We assume that the representation $\rho \colon B \to \B(\mH)$ is nondegenerate. 
Furthermore, consider a unitary $u \in M_n(A\otimes B)$, such that $\rho(u)$ preserves the domain of $\D$, and $[\D,\rho(u)]$ is bounded. 
Then the pairing $[u] \otimes_B [\D] \in \KK^0(\C,A)$ is given under the standard isomorphism $\KK^0(\C,A) \xrightarrow{\simeq} \K_0(A)$ by 
\[
[u] \otimes_B [\D] = \Index\big( P_+(\D) \rho(u) P_+(\D) \big) \in \K_0(A) .
\]
\end{coro}
\begin{proof}
If $\D$ is invertible, the statement is given by \cref{thm:odd_index_pairing}. 
In general, we know that $\D$ has compact resolvents (since the $C^*$-algebra $B$ is unital and the representation $\rho$ is nondegenerate), and therefore the spectrum of $\D$ is a discrete subset of $\R$. Thus there exists $\mu\in\R$ such that $\D-\mu$ is invertible and $[\D-\mu] = [\D]$. 
We note that $P_+(\D) \rho(u) P_+(\D)$ acts on $\Ran P_+(\D) \subset A \otimes \mH^{\oplus n}$, and that its index equals the index of $P_+(\D) \rho(u) P_+(\D) + 1-P_+(\D)$ on $A \otimes \mH^{\oplus N}$. 
Furthermore, $P_+(\D-\mu) - P_+(\D)$ is compact by \cite[Corollary 3.5]{Les05}, and since the index is stable under compact perturbations, we obtain 
\begin{multline*}
\Index\big( P_+(\D) \rho(u) P_+(\D) + 1-P_+(\D) \big) \\
= \Index\big( P_+(\D-\mu) \rho(u) P_+(\D-\mu) + 1-P_+(\D-\mu) \big) . 
\end{multline*}
Thus the desired equality for $\D$ follows from the equality for the invertible $\D-\mu$. 
\end{proof}

% \bibliographystyle{myamsalpha}
% \bibliography{short,bibliography}

\providecommand{\noopsort}[1]{}
\providecommand{\bysame}{\leavevmode\hbox to3em{\hrulefill}\thinspace}
\providecommand{\MR}{\relax\ifhmode\unskip\space\fi MR }
\providecommand{\MRhref}[2]{%
  \href{http://www.ams.org/mathscinet-getitem?mr=#1}{#2}
}
\providecommand{\href}[2]{#2}
\providecommand{\doi}[1]{\href{https://doi.org/#1}{doi:#1}}
\providecommand{\doilinktitle}[2]{#1}
\providecommand{\doilinkbooktitle}[2]{\href{https://doi.org/#2}{#1}}
\providecommand{\doilinkjournal}[2]{\href{https://doi.org/#2}{#1}}
\providecommand{\doilinkvynp}[2]{\href{https://doi.org/#2}{#1}}
\providecommand{\eprint}[2]{#1:\href{https://arxiv.org/abs/#2}{#2}}

\end{document}